\documentclass[pdftex]{amsart}
\usepackage{amsmath}
\usepackage{amsfonts}
\usepackage{amssymb}
\usepackage{ascmac}
\usepackage{layout}
\usepackage{amsthm}

\title[Semiproperness of Namba forcings and Ideals]{On Semiproperness of Namba forcings and Ideals in Prikry extensions}

\address{
Department of Industrial \& Systems Engineering, Faculty of Science and Engineering, Hosei University, 184-8584, Japan}
\email{kenta.tsukura.85@hosei.ac.jp}

\author{Kenta Tsukuura}

\newcommand{\force}{\Vdash}

\theoremstyle{plain}
\newtheorem{thm}{Theorem}[section]

\newtheorem{lem}[thm]{Lemma}
\newtheorem{coro}[thm]{Corollary}
\newtheorem{clam}[thm]{Claim}
\newtheorem{prop}[thm]{Proposition}
\newtheorem{ques}[thm]{Question}

\begin{document}

\maketitle
\begin{abstract}
 In this paper, we study some variations of Namba forcing $\mathrm{Nm}(\kappa,\lambda)$ over $\mathcal{P}_{\kappa}\lambda$ and show that its semiproperness implies $\mathrm{SSR}([\lambda]^{\omega},<\kappa)$. 

 In particular, Prikry forcing at $\mu$ forces Namba forcing $\mathrm{Nm}(\mu^{+})$ is not semiproper. This shows that there is no semiproper saturated ideal in the extensions by Prikry-type forcings. We also show that $[\lambda^{+}]^{\mu^{+}}$ cannot carry a $\lambda^{+}$-saturated semiproper ideal if $\mu$ is singular with countable cofinality. 
\end{abstract}

\section{Introduction}
The notion of a saturated ideal is one of generic large cardinal axioms. The first model of a saturated ideal over $\aleph_1$ is due to Kunen. He established 
\begin{thm}[Kunen~\cite{MR495118} for $\mu = \aleph_0$]\label{kunenideal}
 If $\kappa$ is a huge cardinal then for every regular cardinal $\mu < \kappa$, there is a poset which forces that $\mu^{+} = \kappa$ carries a saturated ideal. 
\end{thm}
Kunen's proof is useful to construct a model with a saturated ideal over $\mathcal{P}_{\mu^{+}}\lambda$. Many saturated ideals, which have strengthnengths of a saturation property, were found by Kunen's method like in~\cite{MR673792} and ~\cite{MR925267}. 
  In~\cite{MR4654825}, the author proved that Kunen's original ideal of Theorem~\ref{kunenideal} is not so strong in the sense of saturation properties. 

But, if $\mu \geq \aleph_1$ then the saturated ideals obtained by collapsing a huge cardinal are proper* (For example, see Proposition~\ref{semipropersaturated}). It seems that properness* of ideals is very strong by the results of~\cite{MR1978224}. Here, by proper*, we mean the properness in the sense of forcing. For an ideal $I$ over $Z$, the properness of $I$ is $Z \not\in I$ in the usual sense. To distinguish properness of ideals and that of forcings, we write ``proper*'' according to~\cite{MR2191239}.

In the above story, $\mu$ is a regular cardinal. We are interested in the case of singular cardinal $\mu$. A model with a saturated ideal over $\mu^{+}$ was given by Foreman~\cite{MR730584}. Then, the ideal is not proper*. Moreover, it is known that $\mathcal{P}_{\mu^{+}}\lambda$ cannot carry a proper* ideal if $\mu$ singular by Matsubara--Shelah~\cite{MR1900548}. 

On the other hand, Sakai obtained a model with a semiproper ideal over $\mathcal{P}_{\mu^{+}}\lambda$ in~\cite{MR2191239}. In Sakai's proof, $\mu$ can be taken as a singular cardinal. Notice that the ideal of Sakai's model is not saturated. We ask
\begin{ques}\label{mainquestion}
 For a singular cardinal $\mu$, can $\mu^{+}$ carry a saturated and semiproper ideal?
\end{ques}
All known models with a saturated ideal over a successor of a singular cardinal were given by using Prikry-type forcing, as far as the author knows. This paper aims to show all known saturated ideals are not semiproper by proving the following theorem. 

\begin{thm}\label{maintheorem3}
 Suppose that $V \subseteq W$ are inner models and $\mu < \lambda$ are cardinals such that:
\begin{enumerate}
 \item $\mu$ and $\lambda$ are regular cardinals in $V$.
 \item $\mu$ is a singular cardinal in $W$ of cofinality $\omega$.
 \item $\lambda$ is a regular cardinal in $W$. 
 \item $(E^{\lambda}_{\mu})^{V}$ is a stationary subset in $W$. 
\end{enumerate}
Then there is no semiproper precipitous ideals over $\mathcal{P}_{\mu^{+}}\lambda$ in $W$. In particular, Prikry-type forcing over $\mu$ (like Prikry forcing and Woodin's modification) forces there is no semiproper precipitous ideals over $\mathcal{P}_{\mu^{+}}\lambda$ for all regular $\lambda \geq \mu^{+}$. 
\end{thm}

Moreover, if we consider an ideal over $[\lambda^{+}]^{\mu^{+}}$ instead of $\mathcal{P}_{\mu^{+}}\lambda$, the answer of Question~\ref{mainquestion} is NO. 
\begin{thm}\label{maintheorem4}
Suppose that $\mu$ is a singular cardinal with cofinality $\omega$ and $\lambda> \mu$ is a regular cardinal. If $[\lambda^{+}]^{\mu^{+}}$ carries a normal, fine, $\mu^{+}$-complete $\lambda^{+}$-saturated ideal $I$ then $I$ is \emph{not} semiproper. 
\end{thm}

The key of the proof of Theorems~\ref{maintheorem3} and~\ref{maintheorem4} is the study of semiproperness of Namba forcing. Namba forcing was introduced by Namba~\cite{MR0297548} as an example of Boolean algebras that is $(\omega,\omega_1)$-distributive but not $(\omega,\omega_2)$-distributive. 

In recent research, Namba forcing is studied in the context of semiproper forcing sometimes. For example, Namba forcing appears as one of the characterizations of dagger principle $(\dagger)$ that was introduced by Foreman--Magidor--Shelah~\cite{MR924672}. 

\begin{thm}[Shelah~\cite{MR1623206} for (1) $\leftrightarrow$ (2), Doebler--Schindler~\cite{MR2576698} for (3) $\to$ (2)]\label{basictheorem}The following are equivalent:
\begin{enumerate}
 \item $(\dagger)$ holds. That is, every $\omega_1$-stationary preserving poset is semiproper.
 \item $\mathrm{SSR}([\lambda]^\omega)$ holds for all $\lambda \geq \aleph_2$.
 \item $\mathrm{Nm}(\lambda)$ is semiproper for all $\lambda \geq \aleph_2$. 
\end{enumerate}
\end{thm}
$\mathrm{SSR}([\lambda]^{\omega})$ and $\mathrm{Nm}(\lambda)$ are semistationary reflection principles and Namba forcing over $\lambda$, respectively. For the definition of them, we refer to Section 2. There is also a local relationship between (2) and (3) above as follows. 
\begin{itemize}
 \item For $\lambda$-strongly compact cardinal $\kappa$, the Levy collapse $\mathrm{Coll}(\aleph_1,<\kappa)$ forces both $\mathrm{SSR}([\lambda]^{\omega})$ and the semiproperness of $\mathrm{Nm}(\lambda)$. 
 \item (Todor\v{c}evi\'{c}~\cite{MR1261218}) If $\mathrm{Nm}(\aleph_2)$ is semiproper then $\mathrm{SR}([\aleph_2]^{\omega})$ holds. In particular, $\mathrm{SSR}([\aleph_{2}]^{\omega})$ holds. 
\end{itemize}
$\mathrm{SR}([\aleph_2]^{\omega})$ is the stationary reflection principle. We are interested in the tolerance between $\mathrm{SSR}([\lambda]^{\omega})$ and the semiproperness of $\mathrm{Nm}(\lambda)$. In this paper, we introduce Namba forcing $\mathrm{Nm}(\kappa,\lambda)$ over $\mathcal{P}_{\kappa}\lambda$ and show the following.
\begin{thm}\label{maintheorem}Suppose that $\mu^{\omega} < \kappa$ for all $\mu \in \kappa\setminus \aleph_2$ and $\aleph_2 \leq \kappa \leq \lambda$ are regular cardinals. The following are equivalent:
\begin{enumerate}
 \item $\mathrm{Nm}(\kappa,\lambda)$ preserves the semistationarity of any subset of $[\lambda]^{\omega}$.
 \item $\mathrm{SSR}([\lambda]^{\omega},{<}\kappa)$.
\end{enumerate}
\end{thm}
$\mathrm{SSR}([\lambda]^{\omega},{<}\kappa)$ is one of variations of semistationary reflection principles introduced by Sakai~\cite{MR2387938}. This is equivalent with $\mathrm{SSR}([\lambda]^{\omega})$ if $\kappa=\aleph_2$. The semiproperness of $\mathrm{Nm}(\kappa,\kappa)$ is also equivalent with that of $\mathrm{Nm}(\kappa)$ (See Lemma~\ref{namba:semiproperchar}). Theorem~\ref{maintheorem} generalizes Todor\v{c}evi\'c's result. We can regard the semiproperness of Namba forcing as one of the reflection principles.

Theorem~\ref{maintheorem} brings us the following observation. The proof of Theorem~\ref{maintheorem3} is almost the same as that of the following theorem.
\begin{thm}\label{maintheorem2}
 Suppose that $V \subseteq W$ are inner models and $\mu < \lambda$ are cardinals such that:
\begin{enumerate}
 \item $\mu$ and $\lambda$ are regular cardinals in $V$.
 \item $\mu$ is a singular cardinal in $W$ of cofinality $\omega$.
 \item $\lambda$ is a regular cardinal in $W$. 
 \item $(E^{\lambda}_{\mu})^{V}$ is a stationary subset in $W$. 
\end{enumerate}
Then $\mathrm{Nm}(\mu^{+},\lambda)$ is \emph{not} semiproper for all regular $\lambda \geq \mu^{+}$ in $W$. In particular, Prikry-type forcing over $\mu$ (like Prikry forcing and Woodin's modification) forces $\mathrm{Nm}(\mu^{+},\lambda)$ is \emph{not} semiproper for all regular $\lambda \geq \mu^{+}$. In particular, $\mathrm{Nm}(\mu^{+})$ is forced to be \emph{non} semiproper\footnote{This also follows from the proof of~\cite[Theorem 5.7]{MR2576698}.}.
\end{thm}

The structure of this paper is as follows. In Section 2, we recall the basic facts of semistationary subsets, Namba forcings, and saturated ideals. In Section 3, we introduce Namba forcing $\mathrm{Nm}(\kappa,\lambda)$ over $\mathcal{P}_{\kappa}\lambda$. We study the basic properties of $\mathrm{Nm}(\kappa,\lambda)$. In Section 4, we prove Theorems~\ref{maintheorem} and~\ref{maintheorem2}. In Section 5, we study the semiproperness of ideals in Prikry-type extensions. The proofs of Theorems~\ref{maintheorem3} and~\ref{maintheorem4} are given in this section. Section 6 is not related to saturated ideals. However, our results can be positioned in the context of strong compactness of cardinals. We will call these principles $(\dagger)$-aspects of strong compactness and conclude this paper with some observations.

\section{Preliminaries}
In this section, we recall basic facts of semistationary subsets, Namba forcings, and saturated ideals. We use~\cite{MR1994835} as a reference for set theory in general. For the topics of Namba forcings and saturated ideals, we refer to~\cite[Section XII]{MR1623206} and ~\cite{MR2768692}, respectively.

Our notation is standard. In this paper, by $\kappa$ and $\lambda$, we mean regular cardinals greater than $\aleph_2$ unless otherwise stated. We use $\mu$ to denote an infinite cardinal. For $\kappa < \lambda$, $E^{\lambda}_\kappa$ and $E^{\lambda}_{<\kappa}$ denote the set of all ordinals below $\lambda$ of cofinality $\kappa$ and $<\kappa$, respectively. We also write $[\kappa,\lambda] = \{\xi \mid \kappa \leq \xi \leq \lambda\}$. By $\mathrm{Reg}$, we mean the class of regular cardinals. For $\subseteq$-increasing finite sequences $s,t \in [\mathcal{P}(X)]^{<\omega}$, by $s \sqsubseteq t$, we means $t$ end-extends $s$ that is, $\forall a \in t \setminus s\forall b \in s(b \subseteq a)$. $\lambda$-strongly compact cardinal $\kappa$ is a regular cardinal in which $\mathcal{P}_{\kappa}\lambda$ carries a fine ultrafilter. 

In the proof of Theorem~\ref{maintheorem3} and~\ref{maintheorem2}, we use the properties of Prikry forcing and Woodin's modifications. The former was introduced in~\cite{prikry} and the letter appeared in~\cite{MR1007865}. 
 \begin{thm}[Prikry~\cite{prikry}]If $\mu$ is a measurable cardinal then there is a poset $P$ with the following conditions:
\begin{enumerate}
 \item $P$ is $\mu$-centered and $P$ preserves all cardinal. 
 \item $P \force \dot{\mathrm{cf}}(\mu) = \omega$. 
 \item If $P \force \omega <\dot{\mathrm{cf}}(\delta)<\mu$ then $P \force \dot{\mathrm{cf}}(\delta) = \check{cf}(\delta)$. 
\end{enumerate} 
 \end{thm}
\begin{proof}
 For (1) and (2), we refer to~\cite{Gitik}. For (3), see \cite[Section 11]{MR1838355}
\end{proof}

 \begin{thm}[Woodin]If $\mu$ is a measurable cardinal and $2^{\mu}= \mu^{+}$ then there is a poset $P$ with the following conditions:
\begin{enumerate}
 \item $P$ is $\mu$-centered. 
 \item $P \force \mu = \aleph_{\omega}$. 
 \item If $P \force \omega <\dot{\mathrm{cf}}(\delta)<\mu$ then $\omega <{cf}(\delta) <\mu$. 
\end{enumerate} 
 \end{thm}
\begin{proof}
 See~\cite{MR1007865}. 
\end{proof}
\begin{lem}\label{prikryextension}
 If $W$ is a generic extension by Prikry forcing over $\mu$ or Woodin's modification over $\mu$. Then, for every regular $\lambda \geq \mu^{+}$, the following holds.
\begin{enumerate}
 \item $\mu$ is a singular cardinal in $W$ of cofinality $\omega$.
 \item $\lambda$ is a regular cardinal in $W$. 
 \item $(E^{\lambda}_{\mu})^{V}$ is a stationary subset in $W$. 
\end{enumerate}
\end{lem}
\begin{proof}
 Easy. 
\end{proof}

\subsection{Semistationary subsets and Namba forcing}

For a set $W \supseteq \omega_1$, a semistationary subset is an $S \subseteq [W]^{\omega}$ such that ${S}^{\mathbf{cl}} = \{x \in [W]^{\omega} \mid \exists y \in S(y \sqsubseteq_{\omega_1} x)\}$ is stationary in $[\lambda]^{\omega}$. By $y \sqsubseteq_{\omega_1} x$, we mean $y \subseteq x \land y \cap \omega_1 = x \cap \omega_1$. 

For a poset $P$, we say that $P$ is semiproper if and only if $\{M \in [\mathcal{H}_{\theta}]^{\omega} \mid \forall p\in M \cap P\exists q \leq p(q$ is $(M,P)$-semigeneric$)\}$ contains a club for all sufficiently large regular $\theta$. $(M,P)$-semigeneric is a condition that forces $M[\dot{G}] \cap \omega_1 = M \cap \omega_1$. The following lemma is well-known. 

\begin{lem}[Shelah~\cite{MR1623206}]
 The following are equivalent:
\begin{enumerate}
 \item $P$ is semiproper.
 \item $P$ preserves the semistationarity of any sets. 
\end{enumerate}
\end{lem}

We also introduce a useful lemma due to Menas to study semistationary subsets. 
\begin{lem}[Menas~\cite{MR0357121}]\label{menas}Let $W$ and $\overline{W}$ be sets with $\omega_1 \subseteq W \subseteq \overline{W}$. 
\begin{enumerate}
 \item  If $C \subseteq [W]^{\omega}$ is a club then the set $\{\overline{x} \in [\overline{W}]^{\omega} \mid \overline{x} \cap W \in C\}$ is a club in $[\overline{W}]^{\omega}$.
 \item If $\overline{C} \subseteq [\overline{W}]^{\omega}$ is a club then the set $\{\overline{x} \cap W \in [{W}]^{\omega} \mid \overline{x} \in \overline{C}\}$ contains a club in $[W]^{\omega}$. 
\end{enumerate}
\end{lem}

\begin{lem}\label{ssr:upward}
 If $S \subseteq [W]^{\omega}$ is semistationary in $[W]^{\omega}$ then $S$ is semistationary in $[\overline{W}]^{\omega}$. 
\end{lem}
\begin{proof}
 This directly follows from Lemma~\ref{menas}.
\end{proof}

$\mathrm{SSR}([\lambda]^{\omega},<\kappa)$ is the statement that claims, for every semistationary subset $S \subseteq [\lambda]^{\omega}$, there is an $R \in \mathcal{P}_{\kappa}\lambda$ with the following properties:
\begin{enumerate}
 \item $\omega_1 \subseteq R \cap \kappa \in \kappa$
 \item $S \cap [R]^{\omega}$ is semistationary. 
\end{enumerate} 
We say that $S$ reflects to $R$ if $(2)$ above holds. To give an witness $R \in \mathcal{P}_{\kappa}\lambda$ of $\mathrm{SSR}([\lambda]^{\omega},{<}\kappa)$, it is enough to give an $R \in \mathcal{P}_{\kappa}\lambda$ in which some semistatoinary $S\subseteq [\lambda]^{\omega}$ reflects to $R$. So $(1)$ can be removed as follows.

 \begin{lem}If a semistationary subset $S \subseteq [\lambda]^{\omega}$ reflects to $R \in \mathcal{P}_{\kappa}\lambda$ then there is an $R' \in \mathcal{P}_{\kappa}\lambda$ such that $S$ reflects to $R'$ and $\omega_1 \subseteq R' \cap \kappa \in \kappa$. 
 
\end{lem}
\begin{proof}
 See~\cite{MR2191239}.
\end{proof}

\begin{lem}\label{ssr:cobounded}
 If a semistationary subset $S \subseteq [\lambda]^{\omega}$ reflects to $R \in \mathcal{P}_{\kappa}\lambda$ then $\{a \in \mathcal{P}_{\kappa} \lambda \mid S \cap [a]^{\omega}$ is semistationary$\}$ is co-bounded. 
\end{lem}
\begin{proof}
 This follows from Lemma~\ref{ssr:upward}. 
\end{proof}

In the proof of Theorem~\ref{maintheorem2}, we use the following theorem. 
 \begin{thm}[Sakai~\cite{MR2387938}]\label{sakai}
 $\mathrm{SSR}([\lambda]^{\omega},{<}\kappa)$ implies $\mathrm{Refl}(E_{\omega}^{\lambda},<\kappa)$. 
\end{thm}
Here, $\mathrm{Refl}(E_{\omega}^{\lambda},<\kappa)$ is the statement that, for every stationary subset $S \subseteq E_{\omega}^{\lambda}$, there is an $\alpha \in E_{<\kappa}^{\lambda}$ such that $S \cap \alpha$ is stationary in $\alpha$. 

Lastly, we conclude this subsection with the definition of Namba forcing. Namba forcing $\mathrm{Nm}(\lambda)$ is the set of all Namba trees. Namba tree is a tree $p \subseteq [\lambda]^{<\omega}$ with the following conditions:
\begin{enumerate}
 \item $p$ has a trunk $\mathrm{tr}(p)$, that is, $\mathrm{tr}(p)$ is the maximal $t \in p$ such that $\forall s \in p(s \subseteq t \lor t\subseteq s)$. 
 \item For all $s \in p$, if $s \supseteq \mathrm{tr}(p)$ then $\mathrm{Suc}(s) = \{\xi < \lambda \mid s{^\frown}\langle \xi\rangle \in p\}$ is unbounded in $\lambda$. 
\end{enumerate}
$\mathrm{Nm}(\lambda)$ is ordered by inclusion. 
Note that we wrote $s{^\frown} \langle \xi \rangle$ to denote $s \cup \{\xi\}$. But if we write $s{^\frown} \langle \xi \rangle$ then we always assume $\max{s} < \xi$. Therefore, $s{^{\frown}} \langle \xi \rangle$ end-extends $a$. We use a similar notation for finite subsets of an ordered set later.

$\mathrm{Nm}(\lambda)$ is $\omega_1$-stationary preserving and that forces $\dot{\mathrm{cf}}(\lambda) = \omega$. Therefore $\mathrm{Nm}(\lambda)$ is semiproper under $(\dagger)$. Namba forcing is a representative of semiproper posets that change uncountable cofinalities in the following sense. 

\begin{thm}[Shelah~\cite{MR1623206}]\label{shelah:namba}
 The following are equivalent:
\begin{enumerate}
 \item $\mathrm{Nm}(\lambda)$ is semiproper.
 \item There is a poset that is semiproper and forces $\dot{\mathrm{cf}}(\lambda) = \omega$. 
 \item $\Phi_{\lambda}$. 
\end{enumerate}
\end{thm}
For the definition of $\Phi_{\lambda}$, see Section 3. 
\subsection{Saturated ideals}

For an ideal $I$ over $Z$, we write $I^{+} = \mathcal{P}(Z) \setminus I$ for the set of $I$-positive sets. $\mathcal{P}(Z) / I$ is a poset $\langle I^{+},\subseteq \rangle$. We say that $I$ is semiproper if $\mathcal{P}(Z) / I$ is semiproper. We say that $I$ is $\lambda$-saturated if $\mathcal{P}(Z) \setminus I$ has the $\lambda$-c.c. For an ideal over $Z = \mathcal{P}_{\kappa}\lambda$, we simply say that $I$ is saturated if $I$ is $\lambda^{+}$-saturated. Note that we can see an ideal over $\kappa$ as an ideal over $\mathcal{P}_{\kappa}\kappa$. 

For a precipitous ideal $I$ over $Z$, by $\mathrm{comp}(I)$, we mean the least $\kappa$ such that $I$ is not $\kappa^{+}$-complete. Then, the critical point of the generic ultrapower mapping $\dot{j}:V \to \dot{M} \subseteq V[\dot{G}]$ is forced to be $\mathrm{comp}(I)$. We say that an ideal $I$ is exactly and uniformly $\kappa$-complete if $\mathrm{comp}(I \upharpoonright A) = \kappa$ for all $A \in I^{+}$. If $Z = \mathcal{P}_{\kappa} X$ or $Z = [X]^{\kappa}$ then every $\kappa$-complete ideal over $Z$ is exactly and uniformly $\kappa$-complete. 

\begin{lem}\label{ultrapowerbasic}Suppose that $I$ is a normal, fine, exactly and uniformly $\mu^{+}$-complete precipitous ideal over $Z \subseteq \mathcal{P}(X)$. Let $\dot{j}$ be a $\mathcal{P}(Z) / I$-name for the generic ultrapower mapping $\dot{j}:V \to \dot{M} \subseteq V[\dot{G}]$. Then the following are forced by $\mathcal{P}(Z) / I$:
\begin{enumerate}
 \item ${^{\mu^{+}}\dot{M}} \subseteq \dot{M} \cap V[\dot{G}]$.
 \item $\dot{j} ``X \in \dot{M}$. If $I$ is $|X|^{+}$-saturated then ${^{|X|}}\dot{M} \cap V[\dot{G}] \subseteq \dot{M}$.
\end{enumerate}
\end{lem}
\begin{proof}
 See~\cite[Section 2]{MR2768692}
\end{proof}

For a later purpose, we see the cofinalities in the generic ultrapower. Let us introduce Shelah's famous result. 
\begin{lem}[Shelah]\label{shelah}
 Suppose that $V\subseteq W$ are inner models. Let $\lambda$ be a regular cardinal in $V$. If $(\lambda^{+})^{V}$ is a cardinal in $W$ then $\mathrm{cf}^{W}(\lambda) = \mathrm{cf}^{W}(|\lambda|)$. 
\end{lem}
\begin{proof}
 See~\cite[Theorem 4.73]{MR2768694}
\end{proof}
We will use Lemma~\ref{ultrapowercofinality} in Section 5. An example of use is computations of cofinalities. For a given saturated ideal over $\mu^{+}$, if $\mu$ is singular with the countable cofinality then $I$ forces $\dot{\mathrm{cf}}(\mu^{+}) = \omega$ by Lemma~\ref{ultrapowercofinality}. 
\begin{lem}\label{ultrapowercofinality}
 Suppose that $I$ is a normal, fine, exactly and uniformly $\mu^{+}$-complete $\lambda^{+}$-saturated ideal over $Z\subseteq \mathcal{P}(X)$. Suppose that $\lambda = |X|$ is a regular cardinal. 
Then the following holds:
\begin{enumerate}
 \item If $Z \subseteq \mathcal{P}_{\mu^{+}}(X)$ then $\mathcal{P}(Z) / I$ forces that $\dot{\mathrm{cf}}(\lambda) = \dot{\mathrm{cf}}(\mu)$. 
 \item If $Z \subseteq [X]^{\mu^{+}}$ and $\lambda$ is a successor cardinal then $\mathcal{P}(Z) / I$ forces that $\dot{\mathrm{cf}}(\lambda^{-}) = \dot{\mathrm{cf}}(\mu)$. 

\end{enumerate}
\end{lem}
\begin{proof}

 By (2) of Lemma~\ref{ultrapowerbasic} and the assumption, $\lambda^{+} = \mathrm{ot}(\dot{j} ``\lambda^{+}) = \dot{j}(\mu^{+})$ in the extension. Therefore $|\lambda^{+}| = \mu$. Note that $\lambda^{+}$ is a cardinal by the $\lambda^{+}$-saturation of $I$. So Lemma~\ref{shelah} shows the required result. \end{proof}

In this paper, we often consider an ideal in some extension. We introduce the preservation theorem of saturated ideals. Foreman proved the conclusion of Lemma~\ref{sat:preservation} in the case of $Z = \mu^{+}$ in~\cite{MR730584} at first but the same thing holds for a loud class. 
\begin{lem}[Foreman~\cite{MR730584}]\label{sat:preservation}
 Suppose that $I$ is a normal, fine, exactly and uniformly $\mu^{+}$-complete $\lambda^{+}$-saturated ideal over $Z\subseteq \mathcal{P}(X)$. Suppose one of the following.
\begin{enumerate}
 \item $Z \subseteq \mathcal{P}_{\mu^{+}}(X)$ and $|X| = \lambda$. 
 \item $Z \subseteq [X]^{\mu^{+}}$ and $|X| = \lambda^{+}$. 
\end{enumerate}
\end{lem}
Then, for every $\mu$-centered poset $P$ forces $\overline{I}$ is a normal, fine, exactly and uniformly $\mu^{+}$-complete $\lambda^{+}$-saturated ideal over $Z$. $\overline{I}$ is $P$-name for the ideal generated by $I$. 
\begin{proof}
 This follows from Foreman's duality theorem~\cite{MR3038554}. 
\end{proof}

Lastly, we recall forcing projections for Proposition~\ref{semipropersaturated}. For posets $P$ and $Q$, a projection $\pi:Q \to P$ is an order-preserving mapping with the property that $q \leq_P \pi(p)$ implies $\exists r \leq_Q p(\pi(r) \leq_P q)$ and $\pi(1_Q) = 1_{P}$.  Whenever a projection $\pi:Q \to P$ is given, for every dense $D$ in $P$, $\pi^{-1}D$ is also dense in $Q$. It follows that $Q \force \pi``\dot{H}$ generates a $(V,P)$-generic filter, where $\dot{H}$ is the canonical name of $(V,Q)$-generic filter. The quotient forcing is defined by $P \force Q / \dot{G} = \{q \in Q\mid \pi(q) \in \dot{G}\}$, ordered by $\leq_P$. Then $P \ast Q /\dot{G}$ is equivalent with $Q$ in the sense of Boolean completion. 

We say that a projection $\pi:Q \to P$ between complete Boolean algebras is continuous if $\prod \pi ``X = \pi(\prod X)$ for all $X \subseteq Q$ with $\prod X \not= 0$. The continuity of projection is useful when we try to analyze quotient forcing $Q / \dot{G}$. The following lemma will be used in a proof of Proposition~\ref{semipropersaturated}. 
\begin{lem}\label{continuous:closed}
 If $\pi:Q \to P$ is a continuous projection between complete Boolean algebras and $Q$ is $\kappa$-closed then $P$ forces $Q / \dot{G}$ is $\kappa$-closed. 
\end{lem}
\begin{proof}
 Note that $P$ has a $\kappa$-closed dense subset. Let $p \force \{\dot{q}_{i} \mid i < \nu\} \in [Q / \dot{G}]^{<\kappa}$ be a descending sequence. By induction, let us construct $p_{i}$ and $q_{i}$ such that 
\begin{enumerate}
 \item $p \geq p_{0} \geq p_1 \geq \cdots $
 \item $\prod_{i}p_{i} \not = 0$.
 \item $p_{i} \force \dot{q}_{i} = q_{i}$. 
\end{enumerate}
We note that $p_{i} \leq \pi(q_{i})$ for each $i$ by (3). Since $Q$ is $\kappa$-closed, $\prod_{i}q_{i} \in Q$. Then the continuity of $\pi$ shows
\begin{equation*}
 \textstyle\prod_{i}p_{i} \leq \prod_{i}\pi(q_{i}) = \pi(\prod_{i}q_{i}).
\end{equation*}
By the definition of $Q / \dot{G}$, $\prod_{i}p_{i}$ forces that $\prod_i q_{i} \in Q / \dot{G}$ and this is a lower bound of $\{\dot{q}_{i} \mid i < \nu\}$, as desired.
\end{proof}

\section{Properties of $\mathrm{Nm}(\kappa,\lambda,I)$}
From here, throughout this paper, we fix regular cardinals $\omega_2 \leq \kappa \leq \lambda$. In this section, for a fine ideal $I$ over $\mathcal{P}_\kappa\lambda$, we introduce $\mathrm{Nm}(\kappa,\lambda,I)$ and study this. Most things in this section are analogies of original $\mathrm{Nm}(\lambda)$. An importance is that $\mathrm{Nm}(\kappa,\lambda,I)$ is $\omega_1$-stationary preserving, this is shown in Lemma~\ref{namba:sttpreserving}. Our proof of this is essentially due to Shelah (See Lemma~\ref{namba:clubpullback}).

For a fine ideal $I$ over $\mathcal{P}_{\kappa}\lambda$, a ($I$-)Namba tree $p$ is a set $p\subseteq [\mathcal{P}_{\kappa}\lambda]^{<\omega}$ with the following conditions:
\begin{enumerate}
 \item $p$ is a tree. That is, each $s \in p$ is $\subseteq$-increasing and $p$ closed under the initial segment. 
 \item $p$ has a trunk $\mathrm{tr}(p)$. 
 \item For each $s \in p$, if $s \sqsupseteq \mathrm{tr}(p)$ then $\mathrm{Suc}(s) = \{a \in \mathcal{P}_{\kappa}\lambda \mid s{^{\frown}}\langle a\rangle\} \in I^{+}$. 
\end{enumerate} 
Let $\mathrm{Nm}(\kappa,\lambda,I)$ be the set of all $I$-Namba trees. $\mathrm{Nm}(\kappa,\lambda,I)$ is ordered by inclusion. we denote $\mathrm{Nm}(\kappa,\lambda,J_{\kappa\lambda}^{bd})$ by $\mathrm{Nm}(\kappa,\lambda)$. $J^{bd}_{\kappa\lambda}$ is the bounded ideal over $\mathcal{P}_{\kappa}\lambda$. 

For $q \in \mathrm{Nm}(\kappa,\lambda,I)$ and $n < \omega$, we write $\mathrm{Lev}_{q}(n) = q \cap [\mathcal{P}_{\kappa}\lambda]^{n}$. This set is the $n$-th levels of $q$. For $p,q \in \mathrm{Nm}(\kappa,\lambda,I)$, by $q\leq^{\ast} p$, we mean $q \leq p$ and $\mathrm{tr}(p) = \mathrm{tr}(q)$. 

\begin{lem}
 If $I$ is a fine ideal over $\mathcal{P}_{\kappa}\lambda$ then, for all $n < \omega$, $D_{n} = \{p \in \mathrm{Nm}(\kappa,\lambda,I) \mid |\mathrm{tr}(p)| \geq n\}$ is a dense subset of $\mathrm{Nm}(\kappa,\lambda,I)$. 
\end{lem}
\begin{proof}
 Easy. 
\end{proof}

$\mathrm{Nm}(\kappa,\lambda,I)$ changes the cofinalities of regular cardinals lying between $\kappa$ and $\lambda$. 
\begin{lem}\label{namba:changecofinality}Suppose that $I$ is a fine ideal over $\mathcal{P}_{\kappa}\lambda$. 
 For all $\delta \in [\kappa,\lambda]\cap \mathrm{Reg}$, $\mathrm{Nm}(\kappa,\lambda,I)$ forces $\mathrm{cf}(\delta) = \omega$.
\end{lem}
\begin{proof}
 Consider a $\mathrm{Nm}(\kappa,\lambda,I)$-name $\dot{g}$ for the set $\bigcup\{\mathrm{tr}(p) \mid p \in \dot{G}\}$. $\dot{g}$ is forced to be a countable $\subseteq$-increasing sequence. Let $\dot{g}_n$ be $\mathrm{Nm}(\kappa,\lambda,I)$-name for the $n$-th element of $\dot{g}$. It is easy to see
\begin{enumerate}
 \item $\force \bigcup_{n}\dot{g}_{n} = \lambda$ and 
 \item $p \force \{\dot{g}_{0},...,\dot{g}_{n-1}\} = \mathrm{tr}(p)$. Here, $n = |\mathrm{tr}(p)|$. 
\end{enumerate}
For each $\delta \in [\kappa,\lambda] \cap \mathrm{Reg}$, by (1), $\force \bigcup_{n} \dot{g}_{n} \cap \delta = \delta$. By (1), $\force \dot{g}_{n} \cap \delta \in (\mathcal{P}_{\kappa}\delta)^{V}$ for each $n < \omega$. Since $\mathrm{cf}(\delta) = \delta \geq \kappa$, $\dot{\xi}_{n} = \sup \dot{g}_{n} \cap \delta$ is forced to be $ < \delta$. Of course, $\force \sup_{n}\xi_n = \delta$, as desired. 
\end{proof}

We need a fusion sequence argument sometime. 
\begin{lem}\label{namba:fusion}Suppose that $I$ is a fine ideal over $\mathcal{P}_{\kappa}\lambda$. For a $\leq$-descending sequence $\langle q_{n} \mid n < \omega \rangle$, suppose that there is a increasing $l:\omega \to \omega$ such that 
\begin{itemize}
 \item $\mathrm{tr}(q_{n})= \mathrm{tr(q_{n+1})}$ for every $n < \omega$. 
 \item For every $n < \omega$, $\mathrm{Lev}_{q_{n}}(l(n)) = \mathrm{Lev}_{q_{n+1}}(l(n))$ for every $n < \omega$. 
\end{itemize}
 Then, $\prod_{n} q_{n} \in \mathrm{Nm}(\kappa,\lambda,I)$. 
\end{lem}
\begin{proof}
 Let $q = \bigcap_{n} q_{n}$ then $q$ is a tree with trunk $\mathrm{tr}(q_{0})$. By the assumption, for every $n < \omega$, let $m$ be a natural number such that $n < l(m)$ then $\mathrm{Lev}_{q}(n) = \mathrm{Lev}_{q_{m}}(n)$. Therefore $q$ is $I$-Namba tree, as desired. 
\end{proof}

We often deal with an $I$-Namba tree that has a form of $p \upharpoonright s$. Lemma~\ref{namba:interpolant} describes what this tree is and its properties.
\begin{lem}\label{namba:interpolant}
Suppose that $I$ is a fine ideal over $\mathcal{P}_{\kappa}\lambda$. For $p \in \mathrm{Nm}(\kappa,\lambda,I)$ and $s \in p$, $p \upharpoonright s = \{t \in p \mid s \sqsubseteq t \lor t \sqsubseteq s\} \in \mathrm{Nm}(\kappa,\lambda,I)$ is an $I$-Namba tree such that 
\begin{enumerate}
 \item $\mathrm{tr}(p\upharpoonright s) = s$.
 \item For every $q \leq p$, if $s \sqsubseteq \mathrm{tr}(q)$ then $q \leq p\upharpoonright s \leq p$.
\end{enumerate}
\end{lem}
\begin{proof}
 Easy. 
\end{proof}

We fix a $\mathrm{Nm}(\kappa,\lambda,I)$-name for a ordinal $\dot{\alpha}$ such that $\force \dot{\alpha} \in w$ for some set $|w| < \kappa$. For an $I$-Namba tree $p$, we say $p$ is $\dot{\alpha}$-bad, if there is no $q \leq^{\ast} p$ that decides the value of $\dot{\alpha}$. 

\begin{lem}\label{namba:badlemma}Suppose that $I$ is a fine ideal over $\mathcal{P}_{\kappa}\lambda$. For every $\mathrm{Nm}(\kappa,\lambda,I)$-name for an ordinal $\dot{\alpha}$ and a set $|w| < \kappa$ with $\force \dot{\alpha} \in w$, if $p \in \mathrm{Nm}(\kappa,\lambda,I)$ is $\dot{\alpha}$-bad then $\{a \in \mathcal{P}_{\kappa}\lambda \mid p\upharpoonright (\mathrm{tr}(p){^\frown}\langle a\rangle)$ is $\dot{\alpha}$-bad $\} \in I^{+}$. 
\end{lem}
\begin{proof}
 We show the contraposition. Since $A = \{a \in \mathcal{P}_{\kappa}\lambda \mid p\upharpoonright (\mathrm{tr}(p){^\frown}\langle a\rangle)$ is not $\dot{\alpha}$-bad $\}\in I^{\ast}$, for each $a \in A$, there are $p_{a} \leq^{\ast} p\upharpoonright (\mathrm{tr}(p){^\frown}\langle a\rangle)$ and $\alpha_{a}$ such that $p_{a} \force \dot{\alpha} = \alpha_{a}$. Since $I$ is $\kappa$-complete, there is a $\beta \in w$ such that $\{a \in A \mid \alpha_{a} = \beta\}\in I^{+}$. Consider a tree $q = \bigcup\{p_{a} \mid a \in A \land \alpha_{a} = \beta\}$. Then $q \leq^{\ast} p$ forces $\dot{\alpha} = \beta$, as desired. \end{proof}

\begin{lem}\label{namba:nobad}If $I$ is a fine ideal over $\mathcal{P}_{\kappa}\lambda$ then, for every $\mathrm{Nm}(\kappa,\lambda,I)$-name for an ordinal $\dot{\alpha}$ and a set $|w| < \kappa$ with $\force \dot{\alpha} \in w$, there is no $\dot{\alpha}$-bad condition. 
\end{lem}
\begin{proof}

 Suppose otherwise. Let $q \in P$ be an $\dot{\alpha}$- bad condition. By Lemmas~\ref{namba:fusion} and ~\ref{namba:badlemma}, We can define $q \leq^{\ast} p$ such that, for all $s \in q$, $\mathrm{Suc}(s) = \{a \in \mathcal{P}_{\kappa} \lambda \mid q \upharpoonright (s^\frown\langle a \rangle)$ is $\dot{\alpha}$-bad $\}$. Then every extension of $q$ does not decide the value of $\dot{\alpha}$. This is a contradiction. 
\end{proof}
To prove Lemmas~\ref{namba:sttpreserving} and~\ref{namba:sttpreserving2}, we need the following lemma. 
\begin{lem}\label{namba:clubpullback}Suppose that $I$ is a fine ideal over $\mathcal{P}_{\kappa}\lambda$. 
 For a set $a$, if $[a]^{\omega}$ has a club subset $D$ of size $< \kappa$ then, for every $p \in \mathrm{Nm}(\kappa,\lambda,I)$, if $p \force \dot{C}\subseteq [a]^{\omega}$ is a club then $\{z \in [a]^{\omega} \mid \exists q \leq p (q \force z \in \dot{C})\}$ contains a club. 
\end{lem}
\begin{proof}
 For each $z \in D$, let us define a two player game $G_{z}$ of length $\omega$ as follows:
 \begin{center}
  \begin{tabular}{c||c|c|c|c|c|c}
   Player I & $X_{0}$, $\xi_0$ & & $\cdots$ & $X_{i},\xi_i$ & & $\cdots$ \\ \hline 
   Player II & & $a_0,y_0$, $q_0$ & $\cdots$ & & $a_i, y_i,q_i$ & $\cdots$
  \end{tabular}
 \end{center}
Players I and II must choose to satisfy the following:
\begin{itemize}
 \item $X_i \in I$. 
 \item $\xi_i \in z$, $y_i \in D$, and $y_i\subseteq y_{i+1}$.
 \item $p \geq q_0 \geq q_1 \geq \cdots$. 
 \item $a_i\in \mathrm{Suc}_{q_i}(\mathrm{tr}(q_i)) \setminus X_i$.
 \item $\mathrm{tr}(q_i) = |\mathrm{tr}(p)| + i$ and $\mathrm{max}(\mathrm{tr}(q_i)) = a_i$. 
 \item $q_i \force \xi_i \in y_i \in \dot{C}$.
\end{itemize}
Note that Player II can choose $q_i$ anytime by Lemma~\ref{namba:nobad}. 
Player II wins if $\bigcup_{i}y_i \subseteq z$. It is easy to see that $G_{z}$ is an open game, and thus, $G_z$ is determined.

First, we claim that $C_0 = \{z \in [a]^{\omega} \mid $ Player II has a winning strategy in $G_z\}$ contains a club subset. Suppose otherwise. Then, there is a stationary subset $T \subseteq D$ such that, for every $z \in T$, Player I has a winning strategy in $G_z$ by the determinacy of games.
 Let $\langle \mathrm{ws}_{z} \mid z \in T\rangle$ be I's winning strategies. Let $\theta$ be a sufficiently large regular cardinal. Consider an elementary substructure $M \prec\mathcal{H}_{\Psi}$ such that 
\begin{enumerate}
 \item $M$ is countable.
 \item $\kappa,\lambda,I, p, \dot{C}, D,\langle \mathrm{ws}_{z} \mid z \in T\rangle \in M$.
 \item $M \cap a = \overline{z} \in T$. 
\end{enumerate}
To prove the contradiction, let us find a sequence of Player II's move $\langle a_i,y_i,q_i \mid i < \omega \rangle$ in which $\langle a_{i},y_i,q_i\rangle \in M$ and this is a regal move after Player I taked $\mathrm{ws}_{\overline{z}}(a_j,y_j,q_j \mid j < i)$. In this play, Player II wins while Player I uses her winning strategy, as we see later. 

Suppose that $\langle a_j,y_j,q_j \mid j < i \rangle \in M$ has been obtained. For each $z \in T$, let $\langle X_i^z, x_i^z \rangle = \mathrm{ws}_{z}(\langle a_j,y_j,q_j \mid j < i \rangle)$ if $\langle a_j,y_j,q_j \mid j < i \rangle$ is a partial play of $G_{z}$ along $\mathrm{ws}_z$. Note that $\langle X_{i}^z ,x_{i}^z \mid z \in T \rangle \in M$ by the induction hypothesis and $\langle \mathrm{ws}_{z} \mid z \in T \rangle \in M$. By $|D| < \kappa$,  $X = \bigcup_{z \in T}X_i^z \in I$. So, in $M$, we can pick $a_i \in \mathrm{Suc}_{q_{i-1}}(\mathrm{tr}(q_{i-1})) \setminus X \subseteq \mathrm{Suc}_{q_{i-1}}(\mathrm{tr}(q_{i-1})) \setminus X_{i}^{\overline{z}}$. Then $q_{i-1} \upharpoonright \mathrm{tr}(q_{i-1}){^{\frown}}\langle a_i \rangle \in M$. By the rule of $G_{\overline{z}}$, $\xi_i^{\overline{z}} \in \overline{z} = M \cap a$. By Lemma~\ref{namba:nobad}, we can choose $q_{i}$ and $y_{i}$ such that $q_{i} \force \xi_i^{\overline{z}} \in y_{i} \in \dot{G}$ and $y_{i-1} \subseteq y_{i}$. The induction is completed. 

Player I used a winning strategy $\mathrm{ws}_{\overline{z}}$ in this play but Player II wins. Indeed, $\bigcup_{i}y_i \subseteq M \cap a = \overline{z}$ by each $y_i$ is in $M$. This is a contradiction. 

 Fix $z \in C_0$. Lastly, we claim that there is a $q \leq p$ that forces $z \in \dot{C}$. Let $\{\xi_0,\xi_1,...,\}$ enumerates $z$.
By induction, let us define a descending sequence $q_0 \geq^{\ast} q_{1} \geq^\ast \cdots$ with the following conditions:
\begin{enumerate}
 \item $q_0 = p$.
 \item $\mathrm{Lev}_{j}(q_i) = \mathrm{Lev}_{j}(q_{i+1})$ for all $j< |\mathrm{tr}(p)| + i$.
 \item For every $s \in \mathrm{Lev}_{|\mathrm{tr}(p)| + i}(q_{i})$, if we write $s \setminus \mathrm{tr}(p) = \{a_0^s,...,a_{i}^s\}$ then there is a $y_i^s$ such that $\langle a^s_i,y^s_i,q^s_i \rangle = \mathrm{ws}(\langle X^s_j,\xi_j \mid j \leq i)$ for some $X^s_0,...,X^{s}_{i} \in I$. 
\end{enumerate}
 Suppose that $q_i$ has been defined. Let us define $q_{i+1}$. For each $s \in \mathrm{Lev}_{|\mathrm{tr}(p)| + i}(q_{i})$, first we consider $q' = q_{i} \upharpoonright s$. Let $X_s$ be the set of all $a \in \mathrm{Suc}_{q'}(s)$ such that $a$ appears in $\mathrm{ws}(\langle X^s_j,\xi_j \mid j < i\rangle{^{\frown}}\langle X,\xi_i\rangle))$ for some $X \in I$. $X_{s}$ need to be an $I$-positive set, that is $X_s \in I^{+}$.

Suppose otherwise, let $\langle a,y,q\rangle = \mathrm{ws}(\langle X^s_j,\xi_j \mid j < i\rangle{^{\frown}}\langle X_s,\xi_i\rangle)$. By the definition of $X_s$, $a \in X_s$. On the other hand, by the rule of $G_z$, $a \not\in X_s$. This is a contradiction. So $X_s \in I^{+}$. For each $a \in X_s$, there are $y_a$ and $q_a$ such that $\langle a,y_a,q_a\rangle = \mathrm{ws}(\langle X^s_j,\xi_j \mid j < i\rangle{^{\frown}}\langle X,\xi_i\rangle)$ for some $X \in I$. Define $q_{i+1}$ by 
\begin{center}
 $\bigcup\{q_{a} \mid s \in \mathrm{Lev}_{|\mathrm{tr}(p)| + i}(q_i)$ and $a \in X_s\}$
\end{center}
Note that $\mathrm{tr}(q_{a}) = \mathrm{tr}(q){^{\frown}} s {^{\frown}} \langle a\rangle$ for each $a \in X_{s}$, $\mathrm{Suc}_{q_{i+1}}(s) = X_{s}$ and $\mathrm{Lev}_{q_{i+1}}(|\mathrm{tr}(p)| + i)=\mathrm{Lev}_{q_{i}}(|\mathrm{tr}(q)| + i)$.

By the definition of $\{q_{i} \mid i < \omega \}$ and Lemma~\ref{namba:fusion}, a lower bound $q = \prod_{i}q_{i} \leq p$ exists. Let $\dot{y}_{i}$ be a $\mathrm{Nm}(\kappa,\lambda,I)$-name such that 
\begin{center}
 $q \force \dot{y}_{i} = y_{i}^{s}$ if and only if $s \sqsubseteq \dot{g}$.
\end{center}
This is well-defined. It is easy to see that $q \force \xi_i \in \dot{y}_{i} \subseteq \dot{y}_{i+1} \in \dot{C}$ for each $i$. Then $q \force \bigcup_{i}{\dot{y}_i} = z \in \dot{C}$, as desired. 
\end{proof}

\begin{lem}\label{namba:sttpreserving}Suppose that $I$ is a fine ideal over $\mathcal{P}_{\kappa}\lambda$.  
 If $S$ is semistationary subset of $[a]^{\omega}$. If $[a]^{\omega}$ has a club subset $D$ of size $<\kappa$ then $\mathrm{Nm}(\kappa,\lambda,I)$ preserves the semistationarity of $S$. In particular, $\mathrm{Nm}(\kappa,\lambda,I)$ is $\omega_1$-stationary preserving.
\end{lem}
\begin{proof}
 Let $p \in \mathrm{Nm}(\kappa,\lambda,I)$ and $\dot{C}$ be such that $p \force \dot{C} \subseteq [a]^{\omega}$ is a club. By Lemma~\ref{namba:clubpullback}, we can take $q \leq p$ and $y \in S$ such that $y\sqsubseteq_{\omega_1} z$ and $q \force z \in \dot{C}$, as desired. 
\end{proof}

\begin{prop}\label{namba:sttpreserving2}
 If $S$ is a stationary subset of $[a]^{\omega}$. If $[a]^{\omega}$ has a club subset $D$ of size $<\kappa$ then $\mathrm{Nm}(\kappa,\lambda,I)$ preserves the stationarity of $S$. In particular, if $\mathrm{cf}(\nu)^{\omega} < \kappa$ and $\mathrm{cf}(\mu) > \omega$ then $\mathrm{Nm}(\kappa,\lambda,I)$ forces ${\mathrm{cf}}(\nu) > \omega$. 
\end{prop}
\begin{proof}
 For the preservation of stationary subsets, Lemma~\ref{namba:clubpullback} works as well as the proof of Lemma~\ref{namba:sttpreserving}. We only check about uncountable cofinalities. Let $X = \{\nu_\xi\mid \xi < \mathrm{cf}(\nu) \}$ be a cofinal subset of $\nu$. Then $[X]^{\omega}$ has a club subset of size $< \kappa$. Therefore, $[X]^{\omega}$ is forced to be a stationary subset. So any countable subset of $X$ in the extension is covered by some element of $([X]^{\omega})^{V}$, as desired.
\end{proof}

In the proof of Theorem~\ref{maintheorem2}, we will use the following lemma:
\begin{lem}\label{namba:semiproperchar}
 The following are equivalent:
\begin{enumerate}
 \item $\mathrm{Nm}(\kappa,\kappa)$ is semiproper. 
 \item $\mathrm{Nm}(\kappa)$ is semiproper.
\end{enumerate}
\end{lem} 

To show Lemma~\ref{namba:semiproperchar}, we need to characterize the semiproperness of $\mathrm{Nm}(\kappa,\lambda,I)$ in terms of game theory. 
Let us introduce the principle $\Phi_{\kappa,\lambda,I}$. $\Phi_{\kappa,\lambda,I}$ is the statement that the player II has a winning strategy for the Galvin game $\Game(I,A)$ for all $A \in I^{+}$. $\Game(I,A)$ is a game of length $\omega$ with two players as follows:
\begin{center}
 \begin{tabular}{c||c|c|c|c|c|c}
  Player I & $F_{0}:A \to \omega_1$ & & $\cdots$ & $F_{i}:A \to \omega_1$ & & $\cdots$ \\ \hline
  Player I &  & $\xi_0 < \omega_1$ & $\cdots$ &  & $\xi_i < \omega_1$ & $\cdots$ 
 \end{tabular}
\end{center}
 Let $\xi = \sup_{n}\xi_{n}$. II wins if $\bigcap_{n < \omega} F_{n}^{-1}\xi \in I^{+}$. We write $\Phi_{\kappa,\lambda}$ for $\Phi_{\kappa,\lambda,J_{\kappa\lambda}^\mathrm{bd}}$. Note that, the Galvin game can be defined for any ideal $I$ over any set $Z$. $\Phi_{\kappa}$ is the statement that Player II has a winning strategy for $\Game(\mu^{+},J_{\kappa}^{\mathrm{bd}})$. $\Phi_{\kappa}$ is equivalent with the semiproperness of $\mathrm{Nm}(\kappa)$ as we saw in Theorem~\ref{shelah:namba}. Lemma~\ref{originalnambaforcing} is an analog of this. Lemma~\ref{namba:semiproperchar} follows by $\Phi_{\kappa,\kappa}\leftrightarrow \Phi_{\kappa}$ (See Lemmas~\ref{scc:char1} and~\ref{scc:char2}). 

Note that $\Phi_{\kappa}$ is a game theoretical variation of Chang's conjecture. For a detail, we refer to ~\cite[Theorem 2.5 of Section XII]{MR1623206}. We see $\Phi_{\kappa,\lambda}$ as a variation of Chang's conjectures. 

\begin{lem}\label{originalnambaforcing}Suppose that $I$ is a fine ideal over $\mathcal{P}_{\kappa}\lambda$. The following are equivalent.
\begin{enumerate}
 \item $\mathrm{Nm}(\kappa,\lambda,I)$ is semiproper. 
 \item $\Phi_{\kappa,\lambda,I}$ holds.
\end{enumerate}
\end{lem}
\begin{proof}
 First, we show the forward direction. Let $C \subseteq [\mathcal{H}_{\theta}]^{\omega}$ be a club such that, for every $M \in C$, 
\begin{itemize}
 \item $M \prec \mathcal{H}_{\theta}$,
 \item $\kappa,\lambda,I \in M$, and,
 \item for every $p \in M \cap \mathrm{Nm}(\kappa,\lambda,I)$, there is a $q \leq p$ that forces $M[\dot{G}] \cap \omega_1 = M \cap \omega_1$. 
\end{itemize}

Consider an expansion $\mathcal{A} = \langle \mathcal{H}_{\theta},\in,C,\kappa,\lambda,I\rangle$. Let us describe a winning strategy for Player II in $\Game(A,I)$. 

  When a function $F:A \to \omega_1$ is given, let $\dot{\alpha}_F$ be a $\mathrm{Nm}(\kappa,\lambda,I)$-name for $\sup F ``\dot{g}$. Let $p_{A}$ be a condition such that $\mathrm{tr}(p_{A}) = \emptyset$ and $\mathrm{Suc}_{p_A}(s) \subseteq A$  for all $s \in p_{A}$. Note that $p_{A} \force \dot{g} \subseteq A$. By Lemma~\ref{namba:sttpreserving}, $p_{A} \force \dot{\alpha}_{F} < \omega_1$. 

 Suppose that Player I played $\langle F_{0},...,F_{i} \rangle$. Let $M_{i}$ be a Skolem hull $\mathrm{Sk}_{\mathcal{A}}(\{p_{A}\}\cup\{F_0,...,F_i\})$. Player II choose $\xi_i = M_i \cap \omega_1$. Then Player II wins. Note that $M = \bigcup_{i}M_i\in C$ by the definition of $M_i$'s. Let $\xi = \sup_{i}\xi_i$. We have $\xi = M \cap \omega_1$. By the choice of $C$, we can take a $q \leq p_{A}$ that forces $\xi = M \cap \omega_1 = M[\dot{G}] \cap \omega_1$. 

 Let $n = |\mathrm{tr}(q)|$ and let $\dot{g}_{n}$ be the $\mathrm{Nm}(\kappa,\lambda,I)$-name for the $n$-th element of $\dot{g}$. We have 
\begin{center}
 $\force F_i(\dot{g}_n) \leq \dot{\alpha}_{F_i} \in M[\dot{G}] \cap \omega_1 = \xi$.
\end{center} 
For every $a \in \mathrm{Suc}_{q}(\mathrm{tr}(q))$, $q\upharpoonright (\mathrm{tr}(q){^\frown}\langle a\rangle)$ forces $\dot{g}_{n} = a$, and thus, $F_i(a)\leq \dot{\alpha}_{F_{i}} < \xi$. So $\mathrm{Suc}_{q}(\mathrm{tr}(q)) \subseteq \bigcap_{i} F^{-1}_{i} \xi$ is $I$-positive, as desired.

 Let us show the inverse direction. We fix a sequence of winning strategies $\langle \mathrm{ws}_{A} \mid A \in I^{+} \rangle$. For a countable $M \prec \mathcal{H}_{\theta}$ with $\langle \mathrm{ws}_{A} \mid A \in I^{+} \rangle \in M$. For $p \in P$, let us find $q \leq^{\ast} p$ which forces $M \cap \omega_1 = M[\dot{G}] \cap \omega_1$. Let $\dot{\alpha}_0, \dot{\alpha}_1,...$ be an enumeration of $\mathrm{Nm}(\kappa,\lambda,I)$-names for a countable ordinal belonging to $M$. 

 We put $t = \mathrm{tr}(p)$. 
 Let $F^{t}_{0}:\mathrm{Suc}(t) \to \omega_1$ be a function such that:
\begin{itemize}
 \item $p\upharpoonright (t^{\frown}\langle a\rangle)$ has a direct extension $p^{0}_{a}$ that forces $\dot{\alpha}_0 = F^{t}_{0}(a)$. 
\end{itemize}
 Let $\xi_0$ be II's first play using $\mathrm{ws}_{\mathrm{Suc}(t)}$ and $A^{t}_0 = ({F^{t}_{0}})^{-1}\{\xi_{0}\}$. Define $p_0 = \bigcup_{a \in A^{0}}p_{a}^{0}$. 

 Next, let us define $p_1$. First, for each $s\in \mathrm{Lev}_{|t|+1}(p_0)$, let $F^{s}_0:\mathrm{Suc}(s) \to \omega_1$ be a function such that

 \begin{itemize}
 \item $p_0\upharpoonright (t^{\frown}\langle a\rangle)$ has a direct extension $p^{s}_{a}$ that forces $\dot{\alpha}_1 = F^{t}_{0}(a)$. 
\end{itemize}
 Let $\xi_0^s$ be II's first play using $\mathrm{ws}_{\mathrm{Suc}(s)}$ and $A^s_{0} = (F^{s}_{0})^{-1}\{\xi_{0}^{s}\}$. Let us define $F^{t}_{1}:\mathrm{Suc}(t) \to \omega_1$ by 
\begin{itemize}
 \item If $a \in A_{0}^{t}$ then $F^{t}_{1}(a) = \xi_{0}^{t^{\frown}\langle a\rangle}$. 
 \item If $a \not\in A_{0}^{t}$ then $F^{t}_{1}(a) = 0$. 
\end{itemize}
 Let $\xi_1^{t}$ be II's second play using $\mathrm{ws}_{\mathrm{Suc}(s)}$ and $A^{t}_{1} = (F^{t}_{1})^{-1}\{\xi_{1}^{t}\}$. Define $p_1 = \bigcup_{a \in A^{t}_1} \bigcup_{b \in A_{0}^{t^\frown\langle a\rangle}}p_{t^\frown\langle a\rangle}^{b}$. Of course, $p_1 \leq^{\ast} p_0$. 

Similarly, we continue this process for all $\dot{\alpha}_n$. Then $q = \bigcap_n p_{n}$ is a tree such that, for all $s \in q$, if $s \sqsupseteq t$ then $\mathrm{Suq}(s) = \bigcap_{n}A_{n}^{s}$. Each of the components $A_n^s$ was a response by II using her winning strategy. Therefore $\bigcap_{n}A_{n}^{s} \in I^{+}$. So $q \in \mathrm{Nm}(\kappa,\lambda,I)$ and that forces $\dot{\alpha}_{n}$ is bounded by $\xi_{n}^{t}$ and it is in $M$. In particular, $q \force M \cap \omega_1 = M[\dot{G}] \cap \omega_1$, as desired.
\end{proof}

\begin{lem}\label{scc:char1}Suppose that $I$ is a fine ideal over $\mathcal{P}_{\kappa}\lambda$.
 If $\mathrm{Nm}(\kappa,\lambda,I)$ is semiproper then $\mathrm{Nm}(\delta)$ is semiproper for all $\delta \in [\kappa,\lambda] \cap \mathrm{Reg}$. 
\end{lem}
\begin{proof}
 By the assumption and Lemma~\ref{namba:changecofinality}, $\mathrm{Nm}(\kappa,\lambda,I)$ is a semiproper poset that changes the cofinality of $\delta$ to $\omega$. By Theorem~\ref{shelah:namba}, $\mathrm{Nm}(\delta)$ is semiproper.
\end{proof}

\begin{lem}\label{scc:char2}
 $\Phi_{\kappa}$ implies $\Phi_{\kappa,\kappa}$. 
\end{lem}
\begin{proof}
 Let $\mathrm{ws}'$ be a II's winning strategy of $\Game(\kappa,J_{\kappa}^{\mathrm{bd}})$. Since $A$ has a subset $\{a_{\xi} \mid \xi < \kappa\} \subseteq A$ such that $\xi < \zeta \to \mathrm{ot}(a_{\xi}) < \mathrm{ot}(a_{\zeta})$ and $\sup_{\xi}\mathrm{ot}(a_{\xi}) = \kappa$. Define an ordinal $\alpha_{\xi} = \mathrm{ot}(a_{\xi})$. For $F:A \to \omega_1$, define $\overline{F}:\kappa \to \omega_1$ by $\overline{F}(\xi) = F(a_{\xi})$. Define II's strategy $\mathrm{ws}$ for $\Game(A,J_{\kappa\lambda}^{\mathrm{bd}})$ by $\mathrm{ws}(F_0,...,F_n) = \mathrm{ws}'(\overline{F}_{0},...,\overline{F}_{n})$. It is easy to see that $\mathrm{ws}$ is a winning strategy. 
\end{proof}

The rest of this section is not related to the main theorems, but we introduce these. The consistency of $\Phi_{\kappa,\lambda,I}$ was shown in \cite{MR0485391}. Recall that $P$ is $\omega+1$-strategically closed if Player II has a winning strategy for the game in which Player I and II alternatively choose $p^{n}$ and $p^{n+1}$ such that $p_0 \geq p_1 \geq \cdots \geq p_{n} \geq p_{n+1} \geq \cdots$. Player II wins if $\prod_{n}p_{n} \not= 0$. The following lemma is essentially due to Galvin--Jech--Magidor~\cite{MR0485391}. 
\begin{lem}\label{strclosed}
 If $I$ is a fine ideal over $\mathcal{P}_{\kappa}\lambda$ such that $\mathcal{P}(\mathcal{P}_{\kappa}\lambda)) / I$ is $\omega + 1$-strategically closed then $\Phi_{\kappa,\lambda,I}$ holds
\end{lem}
\begin{proof}
 Let $\mathrm{ws}'$ be II's winning strategy witnessing $\omega+1$-strategically closedness. We let to describe II's winning strategy $\mathrm{ws}$ for $\Game(A,I)$ for all $A \in I^{+}$. 

 When Player I chooses $F_i:A \to \omega_1$, since $I$ is $\omega_1$-complete, we define $A_i \in I^{+}$ and $\xi_{i}$ such that $F^{-1}_i\{\xi_i\} \cap \bigcap_{j<i}A_j \in I^{+}$. Let $A_i = \mathrm{ws}'( F^{-1}_{k}\{\xi_k\} \cap \bigcap_{j<k}A_j \mid k \leq i)$. Define $\mathrm{ws}(F_0,...,F_{i}) = \xi_i$. 

Then Player II wins. In fact, $\bigcap_{i}A_{i} \subseteq F^{-1} (\sup_{i}\xi_i)$ is in $I^{+}$ since $A_{i}$ is taken by the II's strategy $\mathrm{ws}'$, as desired.
\end{proof}
Galvin--Jech--Magidor proved that $\aleph_2$ carries an ideal $I$ in which $\mathcal{P}(\aleph_2) / I$ is $\omega+1$-strategically closed if a measurable cardinal is collapsed to $\aleph_2$. 

\section{Proof of Theorems~\ref{maintheorem} and~\ref{maintheorem2}}
The first half of this section is devoted to Theorem~\ref{maintheorem}. The rest is about Theorem~\ref{maintheorem2}. 

\begin{lem}\label{main1:lem1}
 Suppose that a semistationary subset $S \subseteq [\lambda]^{\omega}$ does not reflect to any $R \in \mathcal{P}_{\kappa}\lambda$. Then $\mathrm{Nm}(\kappa,\lambda,I)$ destroys the semistationarity of $S$. 
\end{lem}
\begin{proof}
 First, we fix a bijection $\varphi:{^{<\omega}}\lambda \to \lambda$. For each $a \in \mathcal{P}_{\kappa}\lambda$, by the assumption, there is a function $F_{a}:[a]^{<\omega} \to a$ such that there is no $x \in {S}^{\mathbf{cl}} \cap[a]^{<\omega}$ which closed under $F_{a}$.

Let $\dot{g}$ be a ${\mathrm{Nm}(\kappa,\lambda)}$-name for $\bigcup\{\mathrm{tr}(p) \mid p \in \dot{G}\}$. Let $\dot{g}_{i}$ be the $i$-th element of $\dot{g}$. Again, we list the properties of $\dot{g}_i$. 
\begin{enumerate}
 \item $\force \bigcup \dot{g} = \lambda$
 \item $p \force\mathrm{tr}(p) \sqsubseteq \dot{g}$. 
 \item $\dot{g}_{i} \subseteq \dot{g}_{i+1}$. 
\end{enumerate}
We may assume that $\force \omega_1 \subseteq \dot{g}_0$ by shrinking $\mathrm{Suc}_{p}(\emptyset)$. 

 Define a $\mathrm{Nm}(\kappa,\lambda)$-name for a function $\dot{F}:[\lambda]^{<\omega} \to \lambda$ by 
\begin{center}
 $\dot{F}(\xi_{0},...,\xi_{n}) = 
 \begin{cases}
  \varphi(F_{\dot{g}_0}(\xi_0,...,\xi_i),...,F_{\dot{g}_{j}}(\xi_0,...,\xi_i)) & n = 2^{i}3^{j}\\
				  0 & \text{ o.w. }
 \end{cases}$.
\end{center}
Here, $F_{\dot{g}_{k}}(\xi_{0},...,\xi_{i})$ denotes $F_{\dot{g}_{k}}(\{\xi_{0},...,\xi_{i}\} \cap \dot{g}_{k})$ for each $k \leq j$. 

Suppose that $\force \dot{F}$ closed under some $\dot{x} \in {S}^{\mathbf{cl}}$. By Lemma~\ref{menas}, we may assume that $\dot{x}$ can decode $\varphi$. That is, if $\force \varphi(\alpha_0,...,\alpha_n) \in \dot{x}$ then $\force \alpha_{i} \in \dot{x}$. 

 By the choice of $\dot{x}$, we can choose $p \in T$ and $y \in {S}$ such that $p \force y \sqsubseteq_{\omega_1} \dot{x}$. We may assume that $y \subseteq a$ for some $a \in \mathrm{tr}(p)$. Then $p \force |\dot{x} \cap a| = \omega$. We claim that $\dot{x}$ closed under $F_{a}$. 
For every $\xi_{0},...,\xi_{n} \in \dot{x}$, we can choose an expansion $\zeta_{0},...,\zeta_{2^{n}3^{i}-n-1} \in x \setminus (\xi_{n}+1)$ by $p \force |\dot{x} \cap a| = \omega$. $p$ forces 
\begin{center}
 $\dot{F}(\xi_0,...,\xi_{n},\zeta_0,...,\zeta_{2^{n}3^{i}-n-1}) = \varphi(F_{a_0}(\xi_0,...,\xi_{n}),...,F_{a_i}(\xi_0,...,\xi_{n})) \in \dot{x}$. 
\end{center}
And thus, $F_i(\xi_0,...,\xi_{n}) \in x$. Therefore it is forced by $p$ that $\dot{x} \in {S}^{\mathbf{cl}}\cap [a_{i}]^{\omega}$ closed under $F_{a_{i}}$. Let $\mathrm{Hull}(y,F_a)$ denote the least set that contains $y$ as a subset and closed under $F_a$. We have
\begin{center}
 $p \force \mathrm{Hull}(y,F_a) \subseteq \dot{x} \cap a$.
\end{center}
Since $p \force  y \cap \omega_1 = \dot{x} \cap \omega_1$, we also have 
\begin{center}
 $p \force \mathrm{Hull}(y,F_a) \cap \omega_1 = y \cap \omega_1$. 
\end{center}
In particular, $\mathrm{Hull}(y,F_a) \in {S}^{\mathbf{cl}}\cap [a]^{\omega}$ closed under $F_a$. This is contradicting to the choice of $F_{a}$.
 So, the semistationarity of $S$ is destroyed by $p$, as desired.
\end{proof}

\begin{lem}\label{main1:lem2}
 If a semistationary subset $S \subseteq [\lambda]^{\omega}$ reflects to some $R \in \mathcal{P}_{\kappa}\lambda$ then $\mathrm{Nm}(\kappa,\lambda,I)$ preserves the semistationarity of $S$. 
\end{lem}
\begin{proof}
 By Lemma~\ref{namba:sttpreserving}, $\mathrm{Nm}(\kappa,\lambda,I)$ preserves the semistationarity of $S \cap [R]^{\omega}$. By Lemma~\ref{ssr:upward}, the semistationarity of $S$ is also preserved. 
\end{proof}

\begin{proof}[Proof of Theorem~\ref{maintheorem}]
 The forward direction follows from Lemma~\ref{main1:lem1}. The inverse direction follows from~\ref{main1:lem2}.
\end{proof}

We conclude this section with Theorem~\ref{maintheorem2}. 
\begin{proof}[Proof of Theorem~\ref{maintheorem2}]
 It is easy to see that $\mathrm{Refl}(E_{\omega}^{\lambda},{<}\mu^{+})$ fails in $W$. Indeed, $E_{\mu}^{\lambda}$ does not reflect to any ordinals in $E_{<\mu^{+}}^\lambda$. Since $(E_{\mu}^\lambda)^{V} \subseteq (E_{\omega}^{\lambda})^W$ is a stationary subset, this witnesses $\lnot\mathrm{Refl}(E_{\omega}^{\lambda},{<}\kappa)$ in $W$. By Theorem~\ref{sakai},  $\mathrm{SSR}([\lambda]^{\omega},{<}\mu^{+})$ also fails in $W$. By Lemma~\ref{namba:semiproperchar}, $\mathrm{Nm}(\mu^{+},\lambda)$ is not semiproper, as desired. In particular, by Lemma~\ref{originalnambaforcing}, $\mathrm{Nm}(\mu^{+})$ is \emph{non} semiproper in $W$. 
For Prikry-type forcings, Lemma~\ref{prikryextension} works.
\end{proof}

\section{Semiproperness of saturated ideals}
In this section, we discuss about ideals. First, we note that successors of regular cardinals can carry an ideal that is both saturated and semiproper by Proposition~\ref{semipropersaturated}. The rest is devoted to Theorems~\ref{maintheorem3} and~\ref{maintheorem4}. We also study their corollaries. 

\begin{prop}\label{semipropersaturated}If $j:V \to M$ is an almost huge embedding with critical point $\kappa$. For regular cardinals $\aleph_1 \leq \mu < \kappa \leq \lambda < j(\kappa)$, there is a poset which forces the following:
\begin{enumerate}
 \item $\mathcal{P}_{\kappa}\lambda$ carries a saturated and proper* ideal. 
 \item $\mu^{+} = \kappa$ and $\lambda = j(\kappa)$. 
\end{enumerate}
\end{prop}
\begin{proof}
 Let $P = P(\mu,\kappa) \ast \dot{\mathrm{Coll}}(\lambda,<j(\kappa))$. Here, $P(\mu,\kappa)$ is the $\mu$-support diagonal product of Levy collapses $\prod_{\alpha \in [\mu,\kappa) \cap \mathrm{Reg}}^{<\mu}\mathrm{Coll}(\mu,<\kappa)$. For a detail, we refer to \cite[Theorem 1.2]{preprint}. This paper gave a continuous projection $\pi :P(\mu,j(\kappa)) \to P(\mu,\kappa) \ast \dot{\mathrm{Coll}}(\lambda,<j(\kappa))$ and a $P(\mu,\kappa) \ast \dot{\mathrm{Coll}}(\lambda,<j(\kappa))$-name $\dot{I}$ such that $P(\mu,\kappa) \ast \dot{\mathrm{Coll}}(\lambda,<j(\kappa))$ forces the following properties:
\begin{enumerate}
 \item $\dot{I}$ is a saturated ideal over $\mathcal{P}_{\kappa}\lambda$.
 \item $\mathcal{P}(\mathcal{P}_{\kappa}\lambda) / \dot{I} \simeq P(\mu,j(\kappa)) / \dot{G} \ast \dot{H}$.
\end{enumerate}
It is easy to see that $P(\mu,j(\kappa))$ is $\mu$-closed. By Lemma~\ref{continuous:closed}, $P(\mu,j(\kappa)) / \dot{G} \ast \dot{H}$ is forced to be $\mu$-closed. By (2), $\mathcal{P}(\mathcal{P}_{\kappa}\lambda) / \dot{I}$ has a $\mu$-closed dense subset. Since $\mu \geq \aleph_1$, $\mathcal{P}(\mathcal{P}_{\kappa}\lambda) / \dot{I}$ is proper*, as desired.
\end{proof}

A conclusion of Proposition~\ref{semipropersaturated} fails if $\mu$ is singular as follows.
\begin{thm}[Matsubara--Shelah~\cite{MR1900548}]\label{notproper}
 If $\mu$ is a singular cardinal then $\mathcal{P}_{\mu^{+}}\lambda$ cannot carries a proper* ideal. 
\end{thm}
\begin{proof}
 Matsubara--Shelah proved \emph{non}-properness* for a more loud class of ideals, that contains all ideals over successors of singular cardinals. We refer to~\cite{MR1900548}.

 Here, to see a common point between saturated ideals and Namba forcings, we show \emph{non}-properness* assuming the cofinality of $\mu$ is countable and the ideal is saturated. Let $I$ be a saturated ideal over $\mathcal{P}_{\mu^{+}}\lambda$. 
 By (1) of Lemma~\ref{ultrapowercofinality}, $\mathcal{P}(\mathcal{P}_{\mu^{+}}\lambda) / I$ forces $\mathrm{\dot{cf}}(\lambda) = \omega$. Therefore this is not proper* by the same reason of \emph{non}-properness* of Namba forcing.
\end{proof}

 A model in which $\mathcal{P}_{\aleph_{\omega+1}}\lambda$ carries a semiproper ideal $I$ such that $\mathcal{P}(\mathcal{P}_{\aleph_{\omega+1}}\lambda) / I$ forces $\mathrm{cf}(\lambda) = \omega$ was given by Sakai~\cite{MR2191239}. Note that $I$ is not saturated in Sakai's model. 

We study the semiproperness of ideals in terms of the semistationary reflection principle. The following lemma is an analog of Theorem~\ref{maintheorem2} for precipitous ideals. 
\begin{lem}\label{maintheorem2:ideals}
 Suppose that $I$ is a normal, fine, exactly and uniformly $\mu^{+}$-complete precipitous ideal over $Z \subseteq \mathcal{P}(X)$. If $|x| \leq \mu^{+}$ for all $x \in Z$ then the following are equivalent:
\begin{enumerate}
 \item $\mathrm{SSR}([\mu^{+}]^{\omega},{<}\mu^{+})$.
 \item $\mathcal{P}(Z) / I$ preserves all semistationary subset of $[\mu^{+}]^{\omega}$.
\end{enumerate}
\end{lem}
\begin{proof}
Let $\lambda = |X|$. We may assume $X = \lambda$. First, we need the following claim:
\begin{clam}
 There is a $C \in I^{\ast}$ such that $|x \cap \mu^{+}|<\mu^{+}$ for all $x \in C$. 
\end{clam}
\begin{proof}[Proof of Claim]
  Let $G$ be an arbitrary $(V,\mathcal{P}(Z) / I)$-generic and $j:V \to M \subseteq V[G]$ be the generic ultrapower mapping by $G$. Then the following holds:
\begin{itemize}
 \item $j ``\lambda = [\mathrm{id}] \in M$ and ${^{\mu^{+}}}M \cap V[G] \subseteq M$. 
 \item $j(\mu^{+}) \geq \lambda$. 
\end{itemize}
 The first item is well-known. The second item follows from the assumption. 

 Therefore, we have $M \models |\mu^{+}| = |j `` \lambda \cap j(\mu^{+})| < j(\mu^{+})$. By \L\'os's theorem, $\{x \in Z \mid |x \cap \mu^{+}| < \mu^{+}\} \in G$. Since $G$ is arbitrary, $\{x \in Z \mid |x \cap \mu^{+}| < \mu^{+}\} \in I^{\ast}$, as desired. 
\end{proof}
 First, we show the forward direction. Let $S$ be a semistationary subset of $[\mu^{+}]^{\omega}$. By Lemma~\ref{ssr:cobounded}, $D = \{a \in \mathcal{P}_{\mu^{+}}\mu^{+} \mid S \cap [a]^{\omega}$ is semistationary$\}$ is co-bounded. By the claim and Lemma~\ref{menas}, $A =\{x \in C \mid x \cap \mu^{+} \in D\} \in I^{\ast}$. It is easy to see that $A = \{x \in Z \mid S \cap [x \cap \mu^{+}]^{\omega} $ is semistationary$\}$. By \L\'os's theorem, $A$ forces that $(j(S) \cap [j ``\lambda \cap j(\mu^{+})]^{\omega}$ is semistationary$)^{M}$. Since $M$ is forced to be closed under $\mu^{+}$-sequence, $S = j(S) \cap [j ``\lambda \cap j(\mu^{+})]^{\omega}$ is semistationary, as desired.

 Let $S$ be semistationary subset of $[\mu^{+}]^{\omega}$. By the previous argument, we have $A = \{x \in Z \mid |x \cap \mu^{+}| < \mu^{+}$ and $S \cap [x \cap \mu^{+}]^{\omega}$ is semistationary$\} \in I^{\ast}$. Since $I$ is fine and $\omega_1$-complete, we may assume that $\omega_1 \subseteq x \cap \mu^{+}$ for all $x \in A$. Therefore $S$ reflect to $x \cap \mu^{+}$ for all $x \in A$, as desired.\end{proof}
From this, let us show Theorems~\ref{maintheorem3} and~\ref{maintheorem4} .

\begin{proof}[Proof of Theorem~\ref{maintheorem3}]
 By Lemma~\ref{maintheorem2:ideals}, the same proof of Theorem~\ref{maintheorem2} works as well.
\end{proof}
Therefore there is no semiproper saturated ideals over $\mathcal{P}_{\mu^{+}}\lambda$ in the extension by Prikry forcing or Woodin's modification over $\mu$.  Moreover, there is no semiproper precipitous ideals\footnote{It is unknown whether every semiproper ideal is precipitous or not. But, under the GCH, the semiproperness shows the precipitousness, and thus, we can omit the condition of precipitousness from Theorem~\ref{maintheorem3}. For detail, we refer to~\cite{MR2191239}}.

\begin{proof}[Proof of Theorem~\ref{maintheorem4}]
 Since $\lambda$ is a regular cardinal, $S = E_{\lambda}^{\lambda^{+}}$ is a non-reflecting stationary subset. That is, for every $\alpha < \lambda^{+}$, $S \cap \alpha$ is \emph{not} stationary in $\alpha$. Let $G$ be $(V,\mathcal{P}([\lambda^{+}]^{\mu^{+}}))$-generic and $j:V \to M \subseteq V[G]$ be the generic ultrapower mapping induced by $G$. By Lemmas~\ref{ultrapowerbasic} and~\ref{ultrapowercofinality}, $M$ has the following properties:
\begin{enumerate}
 \item ${^{\lambda}M} \cap V[G] \subseteq M$. 
 \item $\lambda^{+} = j(\mu^{+})$. 
 \item $\mathrm{cf}^{V[G]}(\lambda) = \mathrm{cf}^{M}(\lambda) = \omega$. 
\end{enumerate}
Note that $S$ remains a stationary subset since $I$ is $\lambda^{+}$-saturated. 
(1) shows that $S \in M$ and $S$ is non-reflection stationary subset in $M$. By (3), $S$ is forced to be a subset of $E_{\omega}^{\lambda^{+}} = (E_{\omega}^{\lambda^{+}})^{M} = j(E_{\omega}^{\mu^{+}})$. 
Therefore, by (2), 
\begin{center}
 $M \models $ there is a non-reflecting stationary subset of $j(E^{\mu^{+}}_{\omega})$. 
\end{center}
By the elementarity of $j$, we have a non-reflecting stationary subset of $E_{\omega}^{\mu^{+}}$ in the ground. In particular, $\mathrm{SSR}([\mu^{+}]^{\omega},{<}\mu^{+})$ fails by Theorem~\ref{sakai}. By Lemma~\ref{maintheorem2:ideals}, $I$ is \emph{not} semiproper, as desired.
\end{proof}

We have a new mutual inconsistency. 
 \begin{coro}If $\mu$ is a singular cardinal with cofinality $\omega$ then 
 the following are mutually inconsistent:
  \begin{enumerate}
   \item $\mu^{+}$ carries a semiproper ideal.
   \item $[\lambda^{+}]^{\mu^{+}}$ carries a normal, fine, $\mu^{+}$-complete $\lambda^{+}$-saturated ideal for some regular $\lambda$. 
  \end{enumerate}
\end{coro}
\begin{proof}
 If (1) holds then $\mathrm{SSR}([\mu^{+}]^{\omega},{<}\mu^{+})$ holds. By the proof of Theorem~\ref{maintheorem4}, (2) cannot hold.
\end{proof}
It seems that the ``huge-type'' saturated ideal is an anti-compactness principle around singular cardinals in the following sense. 
\begin{coro} 
If $[\lambda^{+}]^{\mu^{+}}$ carries a normal, fine, $\mu^{+}$-complete $\lambda^{+}$-saturated ideal for some regular $\lambda$ and singular $\mu$ of cofinality $\delta$ then there is \emph{no} supercompact cardinal between $\delta$ and $\mu$. 
\end{coro}
\begin{proof}
Suppose otherwise. Let $\kappa$ be a supercompact such that $\delta < \kappa < \mu$. Let $\overline{I}$ be a normal, fine, $\mu^{+}$-complete $\lambda^{+}$-saturated ideal. If $\delta = \kappa$ then $\mathrm{Coll}(\omega_1,<\kappa)$ forces both $(\dagger)$ and the ideal $J$ generated by $I$ is a $\lambda^{+}$-saturated. The former follows from~\cite[Theorem 14]{MR924672}. The latter follows from Lemma~\ref{sat:preservation}

 Note that a saturated ideal is $\omega_1$-stationary preserving. Therefore the ideal is semiproper and $\lambda^{+}$-saturated ideal over $[\lambda^{+}]^{\mu^{+}}$ in the extension. This contradicts to Theorem~\ref{maintheorem4}. 

We assume $\delta$ is uncountable then ${^{<\omega}}\delta$ forces that $\kappa$ is supercompact and the generated ideal $K$ by $I$ is a $\lambda^{+}$-saturated ideal. The latter also follows from Lemma~\ref{sat:preservation}. In the extension, $[\lambda^{+}]^{\mu^{+}}$ carries a $\lambda^{+}$-saturated ideal and $\dot{\mathrm{cf}}(\mu) = |\delta| = \omega$. This is impossible.\end{proof}

Lastly, we point out that the semiproperness of saturated ideals can characterized by the following form:
\begin{prop}
 For a fine ideal $I$ over $\mathcal{P}_{\kappa}\lambda$ with the disjointing property, the following are equivalent:
\begin{enumerate}
 \item $I$ is semiproper.
 \item $\mathrm{Nm}(\kappa,\lambda,I)$ is semiproper.
 \item $\Phi_{\kappa,\lambda,I}$ holds. 
\end{enumerate}
\end{prop}
\begin{proof}
 We only check (3) $\to$ (1). By the disjointing property of $I$, for every countable $M \prec \mathcal{H}_{\theta}$ with $I \in M$ $\mathcal{P}(\mathcal{P}_{\kappa}\lambda)$-name $\dot{\alpha} \in M$ for a countable ordinals, and $A \in I^{+}$, there is an $F:\mathcal{P}_{\kappa}\lambda \to \omega_1$ and $B \in I^{+}$ such that $B \force [F]_{\dot{G}} = \dot{\alpha}$. So, if $M$ has Player II's winning strategy of Galvin game then this strategy shows $C = \bigcap_{F \in M}F^{-1}M \cap \omega_1 \subseteq A$ is $I$-positive. $C$ is a $(M,\mathcal{P}(\mathcal{P}_{\kappa}\lambda))$-semigeneric condition.
\end{proof}

\section{$(\dagger)$-aspects of Strong compactness}
In this section, we introduce the $(\dagger)$-aspects of strong compactness. We begin with Proposition~\ref{recompiledagger}.
\begin{prop}\label{recompiledagger}The following are equivalent
 \begin{enumerate}
  \item $(\dagger)$.
  \item $\mathrm{SSR}([\lambda]^{\omega},{<}\aleph_2)$ for all $\lambda \geq \aleph_2$.
  \item $\mathrm{Nm}(\aleph_2,\lambda)$ is semiproper for all $\lambda \geq \aleph_2$.
 \end{enumerate}
\end{prop}
\begin{proof}
  This follows from Theorems~\ref{basictheorem} and~\ref{maintheorem}.
\end{proof}

\begin{prop}\label{strcpcttonambasp}
 If $\kappa$ is $\lambda$-strongly compact then $\mathrm{Nm}(\kappa,\lambda)$ is semiproper. 
\end{prop}
\begin{proof}
 Note that $\mathcal{P}_{\kappa}\lambda$ carries a fine ultrafilter $U$. Then $\mathcal{P}(\mathcal{P}_{\kappa}\lambda)/ U$ is $\omega+1$-strategically closed by obvious reason. By Lemma~\ref{strclosed}, $\Phi_{\kappa,\lambda,U^{\ast}}$ holds. In particular, $\Phi_{\kappa,\lambda}$ holds too. 
\end{proof}

Recall that our definition of $\kappa$ is $\lambda$-strongly compact is that $\mathcal{P}_{\kappa}\lambda$ carries a fine ultrafilter. On the other hand, the original definition of strong compactness is due to Tarski. He introduced a logic $\mathcal{L}_{\kappa\kappa}$ that admits disjunctions of ${<}\kappa$-many formulas and ${<}\kappa$-many quantifiers. $\kappa$ is strongly compact if, for every $\mathcal{L}_{\kappa\kappa}$-theory $T$, if $T$ is ${<}\kappa$-consistent then $T$ is consistent. Here, by ${<}\kappa$-consistent, we mean that $S$ is consistent for all $S \in [T]^{<\kappa}$. For a detail, we refer to~\cite[Section 4]{MR1994835}. 

There is another formulation of strong compactness. We say that $\kappa$ is $\lambda$-compact\footnote{This phrase is due to Gitik.} if every $\kappa$-complete filter over $\lambda$ can be extended to $\kappa$-complete ultrafilter. It was shown that $\kappa$ is strongly compact if $\kappa$ is $\lambda$-compact for all $\lambda\geq \kappa$. 

Hayut studied these compactnesses. 
\begin{thm}[Hayut~\cite{MR3959249}]
If $\lambda={\lambda^{<\kappa}}$ then the following are equivalent. 
\begin{enumerate}
 \item $\kappa$ is $\lambda$-compact.
 \item For every $\mathcal{L}_{\kappa\kappa}$-theory $T$ with $2^{\lambda}$-many symbols, if $T$ is ${<}\kappa$-consistent then $T$ is consistent. 
\end{enumerate}

\end{thm} (2) is called $\mathcal{L}_{\kappa,\kappa}$-compactness for languages of size $\lambda$. 
Then, Hayut claimed that $\lambda$-compactness affects to cardinals up to $2^\lambda$, contrary to $\lambda$-strong compactness only effects for $\lambda$. In~\cite{MR3959249}, he also drew a picture like
\begin{center}
\begin{tabular}{cccc}
 & $2^\lambda$-strong compactness & $\Rightarrow$ & $\lambda$-compact \\
$\Leftrightarrow$ & $\mathcal{L}_{\kappa,\kappa}$-compactness for language of size $2^{\lambda}$ & $\Rightarrow$ & $\lambda$-strong compactness.
\end{tabular} 
\end{center}

We see a similar phenomenon in the view of semistationary reflection principles and semiproperness of Namba forcings. 
\begin{prop}
If $\alpha^{\omega} < \kappa$ for all $\alpha < \kappa$ then the following are equivalent. 
\begin{enumerate}
 \item $\mathrm{SSR}([\lambda]^{\omega},{<}\kappa)$ for all $\lambda \geq \kappa$.
 \item $\mathrm{Nm}(\kappa,\lambda)$ is semiproper for all $\lambda \geq \kappa$.
\end{enumerate}
\end{prop}
\begin{proof}

  (2) to (1) follows from Theorem~\ref{maintheorem}. (1) to (2) follows from Lemma~\ref{namba:sttpreserving}. Indeed, for every semistationary subset $S\subseteq [\delta]^{\omega}$ for $\delta$, let us check the semistationarity of $S$ is preserved by $\mathrm{Nm}(\kappa,\lambda)$. 

If $\delta < \kappa$ then Lemma~\ref{namba:sttpreserving} works.

 If $\delta \geq \kappa$ then $S$ reflect to some $a \in \mathcal{P}_{\kappa}\delta$. By Lemma~\ref{namba:sttpreserving}, the semistationarity of $S \cap [a]^{\omega}$ is preserved by $\mathrm{Nm}(\kappa,\lambda)$ for all $\lambda$. By Lemma~\ref{ssr:upward}, that of $S$ is also preserved. 
\end{proof}

We let to denote $(\dagger)^{\kappa}$ for convenience by (1) or (2) above holds. $(\dagger)^{\aleph_2}$ is equivalent with $(\dagger)$. Of course, $(\dagger)^{\kappa}$ is equivalent to (1) and (2), respectively. 
\begin{prop}\label{daggerkappa}If $\kappa$ is strongly compact. For an uncountable regular cardinal $\mu < \kappa$, $\mathrm{Coll}(\mu,<\kappa)$ forces $(\dagger)^{\kappa}$. 
\end{prop}
\begin{proof}
 The standard generic elementary embedding argument shows $\mathrm{SSR}([\lambda]^{\omega},{<}\kappa)$ for all $\lambda \geq \kappa$. 
\end{proof}
The following lemma is essentially due to Todor\v{c}evi\'{c}~\cite{MR1261218}. 
\begin{lem}\label{ccsttchar}
 If $I$ is a fine ideal over $\mathcal{P}_{\kappa}\lambda$ then the following are equivalent:
\begin{enumerate}
 \item $\Phi_{\kappa,\lambda,I}$ holds. This is equivalent with $\mathrm{Nm}(\kappa,\lambda,I)$ is semiproper. 
 \item $W_{I}^{A} = \{M \in [\omega_1 \cup {^A}\omega_1]^{\omega} \mid \bigcap\{f^{-1} M \cap \omega_1 \mid f:A \to  \omega_1 \land f \in M\}\in I\} $ is \emph{non}-stationary for all $A \in I$.
\end{enumerate}
\end{lem}
\begin{proof}
 The forward direction is clear. Let us check the inverse direction. We define a winning strategy of Player II for the Galvin game $\Game(I,A)$. Let $\theta$ be a sufficiently large regular cardinal. Let $C \subseteq \mathcal{H}_{\theta}$ be a club subset such that $\{M \cap ({\omega_1 \cup {^{\mathcal{P}_{\kappa}\lambda}}}) \mid M \in C\} \cap W_{I}^{A} = \emptyset$. Consider an expansion $\mathcal{A} = \langle \mathcal{H}_{\theta},\in,A,C,\kappa,\lambda,I,...\rangle$. For given Player I's move $\langle F_{0},...,F_{n}\rangle$, Let $\mathrm{ws}(F_{0},...,F_{n}) = \mathrm{Sk}_{\mathcal{A}}(\{F_{0},...,F_{n}\}) \cap \omega_1$. $\mathrm{Sk}_{\mathcal{A}}$ denotes the Skolem hull. It is easy to see that $\mathrm{ws}$ is a winning strategy. 
\end{proof}
Importance is $W_{I}^A$ is forced to be \emph{non}-semiproper by $\mathrm{Nm}(\kappa,\lambda,I)$ even if $W^A_I$ is stationary as follows. 
\begin{lem}\label{destroysemistt}
 Suppose that $I$ is a fine ideal over $\mathcal{P}_{\kappa}\lambda$ and $A \in I^{+}$. Then $\mathrm{Nm}(\kappa,\lambda,I)$ forces $W_{I}^{A}$ is \emph{non}-semistationary. 
\end{lem}
\begin{proof}
 Let $p_{A}$ be $I$-Namba tree such that 
 \begin{itemize}
  \item $\mathrm{tr}(p_{A}) = \emptyset$. 
  \item $\mathrm{Suc}(s) \subseteq A$ for all $s \in p_A$. 
 \end{itemize}
For every $F:A \to \omega_1$, let us define $\dot{\alpha}_{F}^{n}$ by $p_{A} \force \dot{\alpha}_{F}^{n} = F(\dot{g}_{n})$. Here, $\dot{g}_{n}$ is the $n$-th element of $\dot{g}$. Note that $p_{A}$ force $\dot{g}_{n} \in A$ by the choise of $p_A$. So $\dot{\alpha}_{F}^{n}$ is well-defined. 

 Let $\theta$ be a sufficiently large regular cardinal.
 We claim that if $M \cap (\omega_1 \cup {^{\mathcal{P}_{\kappa}\lambda}\omega_1}) \in W_I$ then $p_{A} \force M[\dot{G}] \cap \omega_1 \not= M\cap\omega_1$ for any $M \in [\mathcal{H}_{\theta}]^{\omega}$. Suppose otherwise, let us fix $q \leq p_{A}$ such that $q$ forces $M \cap \omega_1 = M[\dot{G}] \cap \omega_1$. Putting $s = \mathrm{tr}(q)$, for every $F:A \to \omega_1\in M$, we have
\begin{center}
 $q \force F(\dot{g}_{|s|}) = \dot{\alpha}_{F}^{|s|} \in M[\dot{G}] \cap \omega_1 = M \cap \omega_1 $.
\end{center}
So $\mathrm{Suc}(s) \subseteq \bigcap\{f^{-1}M \cap \omega_1 \mid f :A \to \omega_1\land f \in M\} \in I^{+}$. This is a contradiction. 

Let $\dot{\mathcal{A}}$ be a $\mathrm{Nm}(\kappa,\lambda,I)$-name for an expansion $\langle \dot{\mathcal{H}}_{\theta},\check{\mathcal{H}}_\theta,\in,\kappa,\lambda,A,p_{A},\dot{G},...\rangle$. For a contradiction, we assume that there is a $q \leq p_{A}$ that forces $N \in W_{I}^{A}$ and $N \sqsubseteq_{\omega_1} \dot{M} \prec \dot{\mathcal{A}}$. By the claim above, we have $q$ forces 
\begin{center}
 $N \cap \omega_1 \subsetneq N[\dot{G}] \cap\omega_1 \subseteq \dot{M} \cap \omega_1 = N \cap \omega_1$. 
\end{center}
This is a contradiction. Therefore $W_{I}^{A}$ is forced to be \emph{non}-semistationary, as desired.
\end{proof}

\begin{thm}\label{daggerhierarcy}If $\alpha^{\omega}< \kappa$ for all $\alpha < \kappa$ then the following implication holds:
\begin{center}
\begin{tabular}{cccl}
 &$\mathrm{SSR}([2^{\lambda^{<\kappa}}]^{\omega},{<}\kappa)$ & $\to$ & $\mathrm{Nm}(\kappa,\lambda,I)$ is semiproper for all fine ideal over $\mathcal{P}_{\kappa}\lambda$ \\
 & & $\to$ & $\mathrm{Nm}(\kappa,\lambda)$ is semiproper  \\
& & $\to$ &  $\mathrm{SSR}([\lambda]^{\omega},{<}\kappa)$.
\end{tabular}
\end{center}
\end{thm}
\begin{proof}
The third implication follows from Theorem~\ref{maintheorem}. The second implication is trivial. 

Let us see the first implication. First, we note that $\mathrm{SSR}([2^{{\lambda}^{<\kappa}}]^\omega,{<}\kappa)$ implies that every semistationary subset of size $2^{\lambda^{<\kappa}}$ is preserved by $\mathrm{Nm}(\kappa,\lambda,I)$ by Lemmas ~\ref{ssr:upward} and~\ref{namba:sttpreserving}. By Lemma~\ref{destroysemistt}, $W_{I}^{A}$ needs to be \emph{non}-semistationary for all $A \in I^{+}$. By Lemma~\ref{ccsttchar}, $\mathrm{Nm}(\kappa,\lambda,I)$ is semiproper, as desired. 
\end{proof}
\begin{thm}Assume the same assumption of Theorem~\ref{daggerhierarcy}. 
  The inverse direction of the first and third line of Theorem~\ref{daggerhierarcy} cannot be proven in $\mathrm{ZFC}$ modulo the existence of large cardinals. 
\end{thm}
\begin{proof}
For the first line, let us give a model in which $\mathrm{SSR}([\aleph_2]^{\omega},{<}{\aleph_2})$ but $\mathrm{Nm}(\aleph_2,\aleph_2)$ is \emph{not} semiproper. By the result of Baumgartner~\cite{MR0416925}, Veli\u{c}kovi\'{c}~\cite{MR1174395}, and Sakai~\cite{MR2387938}, it is known that $\mathrm{SSR}([\aleph_2]^{\omega},{<}\aleph_2)$ is equiconsistent with the existence of a weakly compact cardinal. 

On the other hand, If $\mathrm{Nm}(\aleph_2,\aleph_2)$ is semiproper then so $\mathrm{Nm}(\aleph_2)$ is (Lemma~\ref{namba:semiproperchar}). Shelah proved that the semiproperness of $\mathrm{Nm}(\aleph_2)$ implies the Chang's conjecture. So the consistency strength of the semiproperness of $\mathrm{Nm}(\aleph_2)$ is stronger than that of an $\omega_1$-Erd\H{o}s cardinal. 

Therefore, let us assume weakly compact cardinal exists in $L$. In $L$, by the above observations, we can get a model in which $\mathrm{SSR}([\aleph_2]^{\omega},{<}\aleph_2)$ but $\mathrm{Nm}(\aleph_2,\aleph_2)$ is \emph{not} semiproper, as desired.

For the third line, let us give a model in which $\mathrm{Nm}(\kappa,\lambda)$ is semiproper but $\mathrm{SSR}([2^{\lambda^{<\kappa}}]^{\omega},{<}\kappa)$ fails assuming $\kappa$ is a $\lambda$-strongly compact cardinal. We may assume $2^{\lambda^{\kappa}}= \lambda^{+}$. Then $\lambda^{<\kappa} = \lambda$. 
Note that $\mathrm{SSR}([\lambda^{+}]^{\omega},{<}\kappa)$ implies $\lnot\square_{\lambda}$ by Theorem~\ref{sakai}. Let $P$ be the standard poset, that adds $\square_{\lambda}$-sequence, of~\cite{MR1838355}. Note that $P$ does not changes ${^{<\kappa}}(\mathcal{P}(\mathcal{P}_{\kappa}\lambda))$ and thus $\kappa$ remains $\lambda$-strongly compact. By Proposition~\ref{strcpcttonambasp}, the extension is a required model. 
\end{proof}

As we have seen in Proposition~\ref{daggerkappa}, $(\dagger)^{\kappa}$ holds if $\kappa$ is strongly compact or $\kappa$ is Levy collapsed to some successor cardinal. 
Note that another reflection principle holds from the effect of strong compactness. For example, one of them is Rado's conjecture. 

For a tree $T$ of height $\omega_1$, we say that $T$ is a special if there is an $F:T \to \omega$ such that $F^{-1}\{n\} \subseteq T$ is an anti-chain for each $n < \omega$. $\mathrm{RC}(\kappa,\lambda)$ is the following statement: For every non-special tree $T$ of size $\lambda$ has a non-special subtree of size $< \kappa$. 

\begin{thm}
 $\mathrm{RC}(\kappa,\lambda)$ implies $\mathrm{SSR}([\lambda]^{\omega},{<}\kappa)$ holds. 
\end{thm}
\begin{proof}
 The same proof of~\cite[Theorem 5.2]{MR4094551} works as well. 
\end{proof}
It is also known that the inverse direction cannot be proven in $\mathrm{ZFC}$\footnote{For example, see~\cite[Theorem 1.7]{MR3523539}.}. By the contents of this section, under the $\mathrm{GCH}$, we have the following diagram of strong compactness of $\kappa$. 

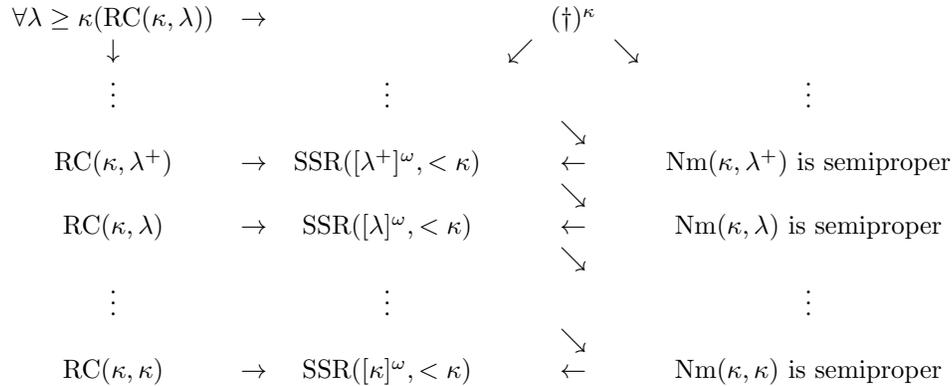
\begin{center}
 \begin{tabular}{ccccc}
 $\forall \lambda\geq \kappa(\mathrm{RC}(\kappa,\lambda))$& $\rightarrow$ & & $(\dagger)^{\kappa}$ & \\

 $\downarrow$&  &  & $\swarrow$\hspace{30pt}$\searrow$ & \\
 $\vdots$&  & $\vdots$ &  & $\vdots$\\
 &  & & $\searrow$ & \\
$\mathrm{RC}(\kappa,\lambda^{+})$ & $\rightarrow$ & $\mathrm{SSR}([\lambda^{+}]^{\omega},<{\kappa})$ & $\leftarrow $ & $\mathrm{Nm}(\kappa,\lambda^{+})$ is semiproper\\
 &  & & $\searrow$ & \\
$\mathrm{RC}(\kappa,\lambda)$ & $\rightarrow$ & $\mathrm{SSR}([\lambda]^{\omega},<{\kappa})$ & $\leftarrow $ & $\mathrm{Nm}(\kappa,\lambda)$ is semiproper\\
 &  & & $\searrow$ & \\
 $\vdots$&  & $\vdots$ &  & $\vdots$\\
 &  & & $\searrow$ & \\
$\mathrm{RC}(\kappa,\kappa)$ & $\rightarrow$  & $\mathrm{SSR}([\kappa]^{\omega},<{\kappa})$ & $\leftarrow $ & $\mathrm{Nm}(\kappa,\kappa)$ is semiproper\\
 \end{tabular}
\end{center}

All of the inverse directions of the arrows above cannot be proven in $\mathrm{ZFC}$ modulo the existence of large cardinals. By this diagram, we can observe the strong compactness in more detail. We conclude this paper with the following question. 
\begin{ques}
 Is it consistent that $\mathrm{Nm}(\kappa,\lambda)$ is semiproper but $\mathrm{Nm}(\kappa,\lambda,I)$ is not semiproper for some fine ideal $I$?
\end{ques}

   \bibliographystyle{plain}
   \bibliography{ref}

\end{document}